\documentclass{article}
\usepackage[dvipdfmx]{graphicx}
\usepackage{amsmath,amssymb,amsthm,tikz-cd}
\theoremstyle{plain}
\newtheorem{theorem}{Theorem}
\newtheorem{lemma}{Lemma}
\newtheorem{prop}{Proposition}
\newtheorem{cor}{Corollary}
\theoremstyle{definition}
\newtheorem{defi}{Definition}
\newtheorem{remark}{Remark}
\begin{document}
\title{The dynamics of hyperbolic rational maps \\ with Cantor Julia sets}
\author{Atsushi KAMEYAMA\thanks{Gifu University}
}
\maketitle
\begin{abstract}
  Let $f:\hat{\mathbb C}\to\hat{\mathbb C}$ be a hyperbolic rational map
  of degree $d\ge2$ on the Riemann sphere.
It has been known that the Julia set $J_f$ is a Cantor set if and only if there exists a
positive integer $n>0$ such that  $\overline{f^{-n}(U)}\subset U$
 for some open topological disc $U$
containing no critical values.
Let $n_f$ denote the minimal positive integer satisfying the above.
The problem is whether  $n_f=1$ or not.

Let $S_d$ denote the shift locus of rational maps of degree $d$.
We show that $n_f=1$ for generic $f\in S_d$ and that there is
a rational map $\bar f\in S_4$ with $n_{\bar f}=2$.
We also prove that $S_d$ is connected using the generic case result.
In particular, generic hyperbolic rational maps of degree $d$
with Cantor Julia sets are
qc-conjugate to each other.
\end{abstract}

\section{Introduction}
In this paper, we investigate the dynamics of rational maps
with Cantor Julia sets.
Specifically, we are interested in rational maps whose dynamics on
the Julia sets are the full shift.
Throughout this paper, we denote by $f:\hat{\mathbb{C}} \to
\hat{\mathbb{C}}$  a rational map of
the Riemann sphere to itself of degree $d\ge2$, by $J=J_f$
its Julia set, and by $F=F_f$ its Fatou set.

In general,  Cantor sets  are often observed in spaces on which several contracting
mappings act.
For example, suppose that $U$ is an open subset of a complete metric space $(X,d)$,
 and that $f_i:\overline U\to \overline U, i=1,2,\dots,N$
are injective mappings with disjoint
images such that $d(f_i(x),f_i(y))\le cd(x,y)$ for some $0<c<1$.
Then there exists a Cantor set $K\subset \overline U$ satisfying $K=\bigcup_{i=1}^N f_i(K)$.
The mapping $g:K\to K$ defined by $g|_{f_i(K)}=f_i^{-1}$ is topologically
conjugate to the one-sided $N$-full shift.

In the context of complex dynamics, the appearance of Cantor Julia sets
for the quadratic family
$q_c(z)=z^2+c$ is well-known.
If $c$ does not belong to the Mandelbrot set, then the Julia set $J_c$ is
a Cantor set, and $q_c|_{J_c}$ is topologically conjugate to the one-sided
$2$-full shift.
Let $G$ be the Green's function for such a  $q_c$, and take $x$ with $G(0)<x<G(c)$.
Then $U=\{z|G(z)< x\}$ is connected and simply connected, and
 $q_c$ has two inverse branches $f_1,f_2$ on $U$.
They satisfy the general setting above for the Poincar\'e metric on $U$.
If $c'$ is another point outside the Mandelbrot set, then clearly $q_c$ and $q_{c'}$
are topologically conjugate to each other on the Julia sets.
It is known that the conjugacy is extended to a qc-conjugacy on the whole sphere.
However this is not the case with general polynomials of degree more than two.

As to rational maps, one of known sufficient conditions
 to have Cantor Julia set is as follows:
 \begin{quote} There exists a (super)attracting
fixed point whose immediate basin of attraction contains all the critical points.
\end{quote}
This is a classical result, due to Fatou \cite{Fatou1919} and Julia \cite{Julia1918}.
We call this condition {\em Cond C}.

For general rational maps, Cond C is not equivalent to having Cantor Julia set;
one of the equivalent conditions  is known as the Branner-Hubbard condition
(\cite{BrHu92},\cite{QiuYin09},\cite{Zhai10}).
However, we use Cond C in this study because we are interested in the situation
where the dynamics on $J$ is topologically
conjugate to the full shift.
If $J$ is Cantor, then there exists a semi-conjugacy from the full shift
(Proposition \ref{propsemiconj}), which is not necessarily homeomorphic.

\begin{defi}
 Let $(\sigma,\Sigma_d)$ be the {\em one-sided $d$-full shift}, namely,
$$\Sigma_d:=\{i_1i_2\cdots:i_k=1,2,\dots,d\}=\{1,2,\dots,d\}^{\mathbb N}$$ is the set of
infinite words of symbols $\{1,2,\dots,d\}$ and the {\em shift map}
 $\sigma:\Sigma_d\to\Sigma_d$ carries $i_1i_2i_3\cdots$ to $i_2i_3\cdots$.
The set $\Sigma_d$ is endowed with the product topology of the discrete space.
\end{defi}

\begin{defi}
 We say
 \begin{enumerate}
\item $f$ is {\em Cantor} if there exists a
homeomorphism $h:\Sigma_d\to J$;
   \item $f$ is {\em d-Cantor (dynamically Cantor)} if  $(f,J)$  is topologically conjugate to $(\sigma,\Sigma_{d})$,
    namely there exists a
   homeomorphism $h:\Sigma_{d}\to J$ such that $f\circ h=h\circ\sigma$;
   \item $f$ is {\em s-Cantor (strongly Cantor)} if there exists
   a closed topological disc $\overline{D}\subset\hat{\mathbb C}$ containing no critical value such that
   ${f^{-1}(\overline D)}\subset \overline{D}$.
   \end{enumerate}
\end{defi}

We raises three questions.
Let $f$ and $g$ be  rational maps of degree $d$ satisfying Cond C.
\renewcommand{\descriptionlabel}[1]{\hspace\labelsep\normalfont #1}
\begin{description}
\item[Q1] Is $f$ d-Cantor?
\item[Q2] Is  $f$ s-Cantor?
 \item[Q3] Are $f$ and $g$ topologically conjugate on $\hat{\mathbb C}$ to each other?
\end{description}

The answers will be given in the present paper.
To sum up,

\begin{description}
  \setlength{\labelwidth}{0mm}
\item[Answer to Q1]
Yes.
We show this result in two ways.
In Theorem \ref{th1}, we choose a collection of $d$ inverse branches of $f$ on $J_f$
which are contracting in some metric.
The other proof is a consequence of the next answer.


\item[Answer to Q2]
No, in general. But yes, generically.
We give an example of a non-s-Cantor rational map $\bar f$ satisfying Cond C (see Figure \ref{fig1}).
In fact, we show the (possibly) stronger statement that any radial
(Definition \ref{deftree}) for this $\bar f$ induces a coding map which is not one-to-one
(Theorem \ref{th3}).
On the other hand, if $f$ with Cond C satisfies
that no two critical orbits meet and  that no critical
orbit lands at a (super)attracting fixed point, then $f$ is s-Cantor (Theorem \ref{th:8}).
This implies that $f$ with Cond C is topologically conjugate on $J_f$ to
the full shift,
since the assumption of Theorem \ref{th:8} is generic in the shift locus
of rational maps.

  \item[Answer to Q3]
  No, in general. But yes, generically.
  It is trivial that two rational maps of the same degree with Cond C are in general
  not topologically conjugate  on the whole sphere to each other.
Meanwhile we show  that the shift locus is connected (Theorem \ref{th6}),
and that if two rational maps satisfy the assumption of
Theorem \ref{th:8} and a generic condition on the
critical points,
then they are qc-conjugate on $\hat{\mathbb C}$ (Theorem \ref{th7}).
These are consequences of Theorem \ref{th:8}, and are not deduced
only from  the general theory in \cite{ManeSadSull83}.
\end{description}

{\centering
\begin{figure}[h]\label{fig1}
  \begin{minipage}{6cm}
    \fbox{\includegraphics[width=5cm]{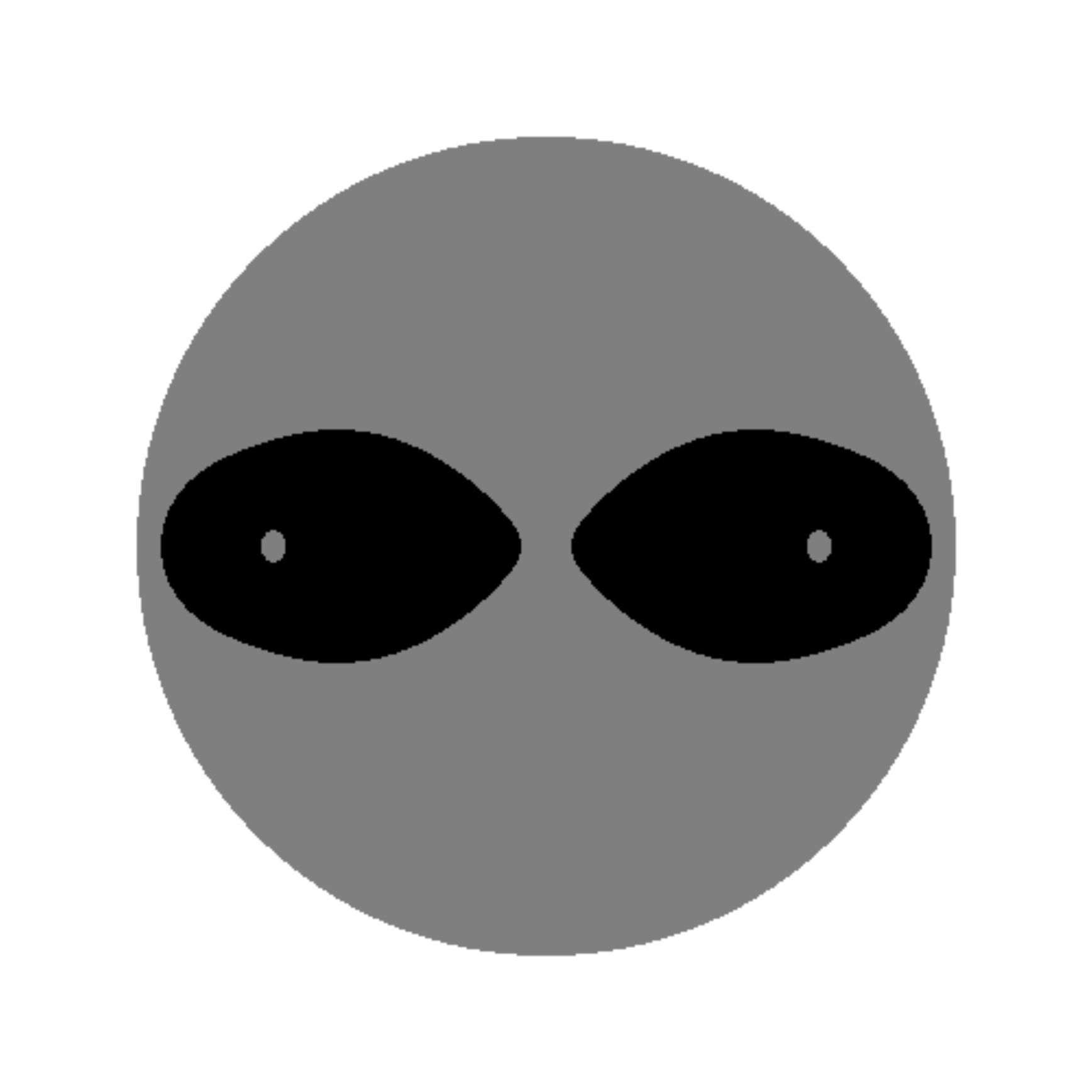}}
  \end{minipage}
  \begin{minipage}{6cm}
    \fbox{\includegraphics[width=5cm]{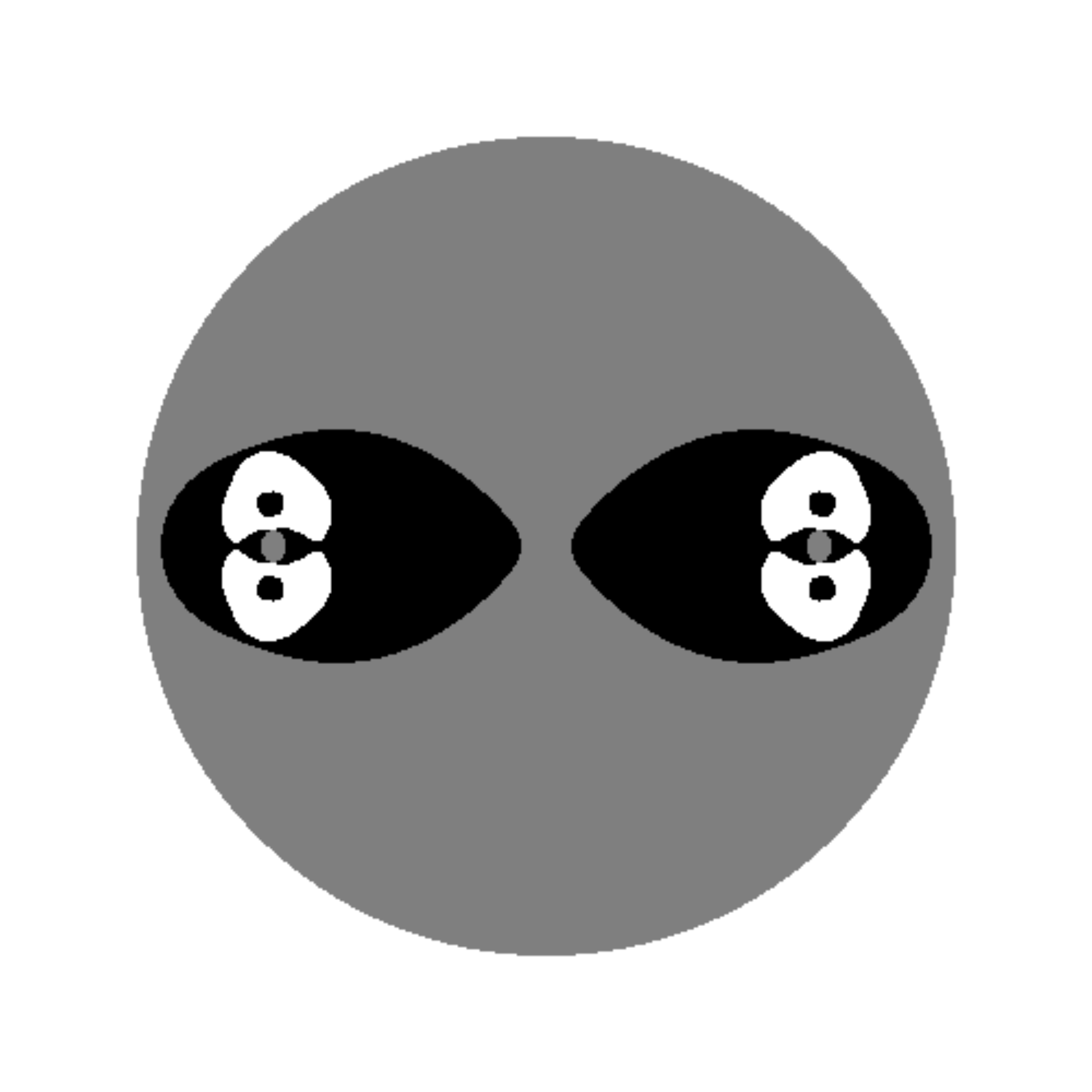}}
  \end{minipage}
  \caption{An example of rational maps d-Cantor but non-s-Cantor.
  Let $\bar f(z)=a(z^2-1)+(4a(z^2-1))^{-1}, a=1.665i$.
  In the left, the gray disc is $D=\{|z|\le 3/2\}$, and the union of
  the black
  annuli  $X_1,X_2$ is $\bar f^{-1}(D)$.
  In the right, we see the inverse image $\bar f^{-1}(X_i)$, which is the union of
  two white annuli $X_{1i}\subset X_1$ and $X_{2i}\subset X_2$.
  We have $\bar f:X_{ji}\to X_i$ is a covering of degree two.
  In Theorem \ref{th3}, we will prove that $\bar f$ is not s-Cantor.}
\end{figure}
}

There exist a lot of researches on the shift locus for polynomials.
Among them,  \cite{BrHu88} and \cite{BlDeKe91} are the pioneer works.
The shift locus for rational maps of degree 2 is investigated in
\cite{GoKe90}.
The connectedness of the shift locus for polynomials is known.
A proof of this fact appears in \cite{DeMarcoPilgrim11}.
However the connectedness of the shift locus for general rational maps
has not been known so far.

\medskip

Our main results  are stated as follows.

In Section 3, we restrict $f$ to be geometrically finite or hyperbolic.
We give many equivalent conditions for $f$ to be d-Cantor.
Among them,

\medskip

\noindent
{\bfseries Theorem A.} (Partial statements of Theorem \ref{th1})
Suppose that $f$ is a hyperbolic rational map.
Then the following are equivalent:
\begin{itemize}
  \item $f$ is Cantor.
  \item $f$ is d-Cantor.
  \item $f^n$ is s-Cantor for some integer $n>0$.
  \item There exists a (super)attracting fixed point such that
  the immediate basin of attraction includes at least $2d-4$
  critical values counted with multiplicity.
  \item Any `bounded iterated monodromy group' is a finite group.
\end{itemize}

\medskip

In Section 4 and 5, we consider a sufficient condition to be s-Cantor, and
deduce some knowledge about the shift locus of rational maps.

\medskip

\noindent
{\bfseries Theorem \ref{th:8}.}
If a Cantor hyperbolic rational map $f$ satisfies that no two critical orbits
meet and that the orbit of any non-fixed critical point
 contains no fixed point, then $f$ is s-Cantor.

\medskip

To prove the theorem, we construct a topological disc $D$ which
satisfies the condition of the definition of s-Cantor.
We start with a topological disc which is the whole sphere minus
a small neighborhood of the (super)attracting fixed point, and deform it combinatorially
several times to obtain a required $D$.
If $d>2$, such $D$ is not unique up to homotopy relative to the postcritical set.

\begin{defi}
 The {\em shift locus} $S_d$ is the set of d-Cantor
 hyperbolic rational maps
 of degree $d$.
 The {\em strong shift locus} $T_d$ is the set of s-Cantor hyperbolic
 rational maps.

 Clearly, $T_d\subset S_d$.

 Let $\mathrm{Pol}$ denote the set of polynomial maps.
\end{defi}

\noindent{\bfseries Theorem B.}(Theorem \ref{th:pol}, \ref{th4}, \ref{th6},
and \ref{th7})
\begin{enumerate}
\item $S_d$ ($d\ge2$) is connected.
\item $T_d\cap\mathrm{Pol}=S_d\cap\mathrm{Pol}$ ($d\ge2$),
$T_2=S_2$, and $T_d\subsetneq S_d$ ($d\ge4$).
\item If $f$ and $g$ satisfy the assumption of
Theorem \ref{th:8} and a generic condition on the
critical points,
then they are qc-conjugate on $\hat{\mathbb C}$

\end{enumerate}

The proof of the last statement of 2 is deferred
to Section 5.
We prove:

\medskip
\noindent
{\bfseries Theorem \ref{th3}.}
There exists a hyperbolic rational map $\bar f$ of degree four which is Cantor
 but not s-Cantor.

\medskip
There we  show that $\bar f$ in Figure \ref{fig1} is not s-Cantor and $\bar f^2$ is s-Cantor.
Our proof of Theorem \ref{th3} is algebraic.
From $\bar f$, we extract a family of homomorphisms $\tilde f_r:\mathbb Z_2\oplus \mathbb Z_2
\to  D_2\times D_2$, where $D_2$ is the Klein four-group.
To verify the statement, we investigate the behavior of the homomorphisms.

 The following problems are open:

 \medskip
 \noindent
 {\bfseries Unsolved Problem 1}
 \begin{itemize}
   \item Does there exist a Cantor hyperbolic rational map of degree three which is
   not s-Cantor?
   \item For any integer $n>0$, does there exist a Cantor hyperbolic rational map $f$
   such that $f^{n}$ is  not s-Cantor and $f^{n+1}$ is s-Cantor?
 \end{itemize}

\section{Preliminary}

We give several notations, definitions, and basic facts.
Let $f:\hat{\mathbb C}\to\hat{\mathbb C}$ be a rational map of
 degree $d$ with Julia set $J=J_f$.

\begin{defi}
  For $z\in\hat{\mathbb C}$, the orbit of $z$ is
  $\mathtt{Orb}(z):=\{f^k(z):k\ge0\}$.
Let $\mathtt{Cr}=\mathtt{Cr}_f$ denote the set of critical points of $f$, and set
$P=P_f:={\{f^k(c):c\in \mathtt{Cr},k>0\}}=\bigcup_{c\in\mathtt{Cr}}\mathtt{Orb}(f(c))$.
The closure of $P$ is said to be the {\em postcritical set}.
For $c\in\mathtt{Cr}$, the orbit $\mathtt{Orb}(f(c))$ is
called a {\em critical orbit}.

Let $\mathtt{At}=\mathtt{At}_f$ denote the set of attracting periodic points
 of $f$, and
 $\mathtt{Pa}=\mathtt{Pa}_f$ the set of parabolic periodic points.

 We say  that $f$ is {\em  geometrically finite} if $J\cap\overline P$
 is finite.
 In other words, each critical orbit is either eventually periodic,
 or converges
 to an attracting or parabolic periodic orbit.
\end{defi}

\begin{defi}
We use the symbol $\sigma $ also as the shift map on the set of finite words
 $\mathtt{Word}(k):=\{i_1i_2\cdots i_k:i_m=1,2,\dots,d\}$.
Each finite word $w=i_1i_2\cdots i_k\in \mathtt{Word}(k)$ can be considered as the branch of $\sigma^{-k}$, namely,
$w:\mathtt{Word}(n)\to \mathtt{Word}(n+k), j_1\cdots j_n\mapsto wj_1\cdots j_n$.
The empty word is denoted by $\emptyset$, or $\mathtt{Word}(0)=\{
\emptyset\}$.
We write $\mathtt{Word}:=\bigcup_{k=0}^\infty \mathtt{Word}(k)$
\end{defi}

\begin{defi}\label{deftree}
Let $l_i:[0,1]\to \hat{\mathbb C}-P$ ($i=1,2,\dots,d$) be  continuous
paths with the initial point $\bar x=l_i(0)$ in common.
We say that a collection of paths
$r=(l_i)_{i=1,2,\dots,d}$ is a
  {\em radial} if
  $f^{-1}(\bar x)\supset  \{l_i(1):i=1,2,\dots,d\}$.
Set $R:=\bigsqcup_{i=1}^d[0,1]_i/\sim$, where $[0,1]_i$
  are copies of the unit interval $[0,1]$ and we glue them at $0$.
  (Recall that $R$ is called the complete bipartite graph $K_{1,d}$
  in graph theory.)
  It is convenient to  consider a radial to be the mapping $r:
R\to \hat{\mathbb C}-P$
with $r(t)=l_i(t)$ for $t\in[0,1]_i$.
Note that a continuous mapping  $r:
  R\to \hat{\mathbb C}-P$ is a radial
  if and only if  $f^{-1}(r(0))\supset \{r(1_i):i=1,2,\dots,d\}$.

  In the situation above,
  we call  $\bar x\in\hat{\mathbb C}-P$ the {\em basepoint}.

For a path $\gamma:[0,1]\to \hat{\mathbb C}-P$ with $\gamma(0)=\bar x$,
 let $L_i(\gamma)$ be the lift of
  $\gamma$ by $f$ (i.e. $f\circ L_i(\gamma)=\gamma$) with
  $L_i(\gamma)(0)=l_i(1)$.
  For $w\in \mathtt{Word}(k)$, we inductively define paths $l_w$ such that
  $l_{iw}=l_iL_i(l_w)$.
  Naturally $L_w$ is defined for any $w\in\mathtt{Word}$,
   and $l_{ww'}=l_wL_{w}(l_{w'})$.
For $\omega=i_1i_2\cdots\in\Sigma_d$, we write $\phi_r(\omega)=\lim_{k\to\infty}
 l_{i_1i_2\cdots i_k}(1)$  if exists.
 We can naturally define an arc $l_\omega$ connecting
 $\bar x$ and $\phi_r(\omega)$ up to a change of parameter.
We say that $\phi_r:\Sigma'\to J$ is the {\em coding map} associated with $r$, where $\Sigma'$ is the set of $\omega\in \Sigma_d$ for which $\phi_r(\omega)$ exists.
This method of coding is originated by F. Przytycki \cite{Prz85},\cite{Prz86},
and is called geometric coding tree.
  In many cases,  $\phi_r$ is continuous and $\Sigma'=\Sigma_d$,
   in particular whenever $f$ is geometrically finite
  (Lemma \ref{lemmetric}) or $f$ is Cantor (Proposition \ref{propsemiconj}).

  We may consider deformation of radials.
  Let $r=(l_i),r'=(l_i')$ be two radials which do not necessarily
  satisfy $r(0)=r'(0)$.
We say that $r$ and $r'$ are {\em homotopic} if
there exists a homotopy $h:
R\times[0,1]\to \hat{\mathbb C}-P$
between $r$ and $r'$ such that  $h(\cdot,s):
  R\to \hat{\mathbb C}-P$ is a radial
 for every $s\in[0,1]$.
 \end{defi}

\begin{defi}
 We say
\begin{enumerate}
  \item $f$ is {\em t-Cantor (tree Cantor)} if there exists a
   radial $r$ such that the coding map $\phi_r:\Sigma_d\to J$ is
   a homeomorphism;
\item  $f$ is {\em ss-Cantor} (semi-strongly Cantor) if
   there exists a path-connected set $X\subset\hat{\mathbb{C}}$ such that
   $f^{-1}(X)$ has $d$ connected components
   $X_1,X_2,\dots,X_d\subset X$, and such that
   $\overline {X_i}\cap\overline{X_j}$ has
   no point of $J$ for
   any $i\ne j$.
  \end{enumerate}

\end{defi}

Now we have five notions of Cantor Julia sets.
It is easy to see that `s-Cantor' $\Rightarrow$ `ss-Cantor' $\Rightarrow$
`t-Cantor' $\Rightarrow$
 `d-Cantor' $\Rightarrow$ `Cantor' (Proposition \ref{lemcantor}).
  It will be proved that the last three are different each other.
  If $f$ is d-Cantor, then $f$ is geometrically finite (Proposition \ref{propgeof}).

  \medskip
  \noindent
  {\bfseries Unsolved Problem 2}
\begin{itemize}
  \item Does it hold that `ss-Cantor' $\Rightarrow$ `s-Cantor'?
  \item Does it hold that `t-Cantor' $\Rightarrow$ `ss-Cantor'?
  \end{itemize}

\medskip

  If $f$ is  Cantor, the Fatou set $F$ is
  connected (see Theorem 13.4 \cite{Moise77}).
Thus a Cantor rational map has neither Siegel disc  nor Herman ring, and
$\#(\mathtt{At}\cup\mathtt{Pa})=1$.
The unique fixed point in $\mathtt{At}\cup\mathtt{Pa}$ is
(super)attracting or parabolic with the
local form  $z\mapsto z+a z^2+\cdots$.

\begin{defi}\label{defidomain}
 Let $p\in\mathtt{At}\cup\mathtt{Pa}$, let $F_0$ be a connected
 component of the Fatou set $F$ with $p\in\overline {F_0}$, and
let $n$ be the period of $F_0$.
 A connected and simply connected open set $U$ included in $F_0$
satisfying the following condition is called
a {\em simple domain} for $p$ and $F_0$:
$p\in \overline U$, $\partial U\cap \overline P\subset \{p\}\cap\mathtt{Pa}$,
$\partial U$ is a piecewise smooth simple
 closed curve, $\overline{f^n(U)}\subset U\cup\{p\}$,
 $\{f^{kn}|U\}_{k>0}$
converges to $p$ uniformly, and for any $x\in F_0$
there exists $k>0$ such that $f^{kn}(x)\in U$.
If $p$ is parabolic, we take $U$ to be a connected component of
an attracting petal.
Remark that $p\in U$ if $p\in\mathtt{At}$, and   $p\notin U$ if $p\in\mathtt{Pa}$

We denote by $U_k$  the connected component of $f^{-kn}(U)$
including $U$.
\end{defi}

\begin{lemma}\label{lemdomain}
 Let $p\in\mathtt{At}\cup\mathtt{Pa}$, and let $F_0$ be a connected
 component of $F$ with $p\in\overline{F_0}$, and
$U$ a simple domain for $p$ and $F_0$.
Take $U_k$ as above.
Then $F_0=\bigcup_{k=1}^\infty U_k$.
In particular, $P\cap F_0\subset U_k$ for some $k>0$.
\end{lemma}

\begin{proof}
Assume that $F_0\ne\bigcup_{k=1}^\infty U_k$.
 Then $F_0\cap(\overline{\bigcup_{k=1}^\infty U_k}-\bigcup_{k=1}^\infty U_k)
 \ne\emptyset$.
 Indeed, otherwise, $\bigcup_{k=1}^\infty U_k$ is open and closed in $F_0$.
 Existence of
 $x \in F_0\cap(\overline{\bigcup_{k=1}^\infty U_k}-\bigcup_{k=1}^\infty U_k)
 \ne\emptyset$ contradicts the fact $f^{kn}(x)\in U$ for some $k>0$.
\end{proof}

\begin{lemma}\label{lemmetric}
 If $f$ is geometrically finite, then $\phi_r(\omega)$ converges
 for any radial $r:R\to\hat{\mathbb C}-\overline P$ and
 any $\omega\in \Sigma_d$.
Moreover,  $\phi_r=\phi_{r'}$ if $r$ and $r'$ are homotopic.

\end{lemma}

 \begin{proof}
  This is a consequence of expansiveness of $f$ on  $\Omega$,
  where $\Omega\subset \hat{\mathbb C}$ is a subset defined below.
The crucial part of the proof is due to \cite{TanYin96}.
We show the outline.

  Let $r=(l_i):R\to\hat{\mathbb C}-\overline P$ be a radial
  with basepoint $\bar x$.
  Let $W$ be a union of simple domains of every $p\in\mathtt{At}
  \cup\mathtt{Pa}$ and every connected component $F_0\subset F$ with
  $p\in\overline{F_0}$ such that $f(W)\subset W$ and $r(R)\cap W
  =\emptyset$.
Then  the compact set
$\Omega=\hat{\mathbb C}-W$ satisfies the following:
$J\subset\Omega$, $f^{-1}(\Omega)\subset\Omega$, and
$\Omega$ is a connected and finitely connected domain which has
the boundary  expressed as a union of finitely many smooth curves.

Let $\nu:\Omega\to\mathbb N$ be the ramification
function for $f|_\Omega$, namely $\nu$ is the minimal
function such that $\nu(z)=1$ for
$z\notin P$ and $\nu(z)$ is a multiple of $\nu(w)\deg_wf$
 for $z\in P,w\in f^{-1}(z)$.
Let $p:S\to(\mathrm{int}\,\Omega,\nu)$ be the universal covering
for the Riemann surface orbifold $(\mathrm{int}\,\Omega,\nu)$,
where $S=D=\{|z|<1\}$ with the Poincar{\'e} metric or $S\subset\mathbb C$
with the Euclidean metric according to the
Euler characteristic of $(\mathrm{int}\,\Omega,\nu)$.

  Fix a point $\hat x\in p^{-1}(\bar x)\subset S$.
Let $\hat l_i$ be the lift of $l_i$ by $p$ with $\hat l_i(0)=\hat x$.
  There exist holomorphic maps
$g_i:S\to S$ such that $f\circ p\circ g_i=p$ and
$g_i(\hat x)=\hat l_i(1)$ for $i=1,2,\dots,d$.
If $\mathtt{Pa}=\emptyset$, then there exists $0<\alpha<1$
 such that $\rho(g_i(x),g_i(y))<\alpha \rho(x,y)$, where
$\rho$ is the distance determined by the Poincar{\'e} or Euclidean
metric on $S$.
If $\mathtt{Pa}\ne\emptyset$, then $S=D$ and we extend
$p:D\to\mathrm{int}\,\Omega$ to a surjection
$p:\hat D\to \Omega$,
where $\hat D\subset \overline D$ and any continuous arc
$c:[0,1]\to\Omega$ can be lifted to $\hat c:[0,1]\to \hat D$.
The maps $g_i:D\to D$ are also extended to
$g_i:\hat D\to\hat D$.
Modifying the Poincar{\'e} metric, we can construct
a metric on $\hat D$ which determines the distance $\rho$
such that there exists an increasing function
$h:[0,\infty)\to[0,\infty)$  with $h(t)<t$ for $t>0$ and
$\rho(g_i(x),g_i(y))<h(\rho(x,y))$.

Then we have a compact set $K\subset \hat S$ ($\hat S=S$
or $\hat D$) such that $$K=\bigcup\limits_{i=1,\dots,d}g_i(K).$$
Namely, $K$ is the attractor of the iterated function system
 $(\{g_i\},\hat S)$.
  By construction, $p(K)= J$.
  Moreover,
  $p\circ g_{i_1}\circ g_{i_2}\circ\cdots\circ g_{i_k}(\hat x)
  =l_{i_1i_2\cdots i_k}(1)$ for $i_1i_2\cdots i_k\in \mathtt{Word}(k)$.
Canonically, we have a coding map $\pi':\Sigma_d\to K$ and
  $\phi_r=p\circ \pi':\Sigma_d\to J$.

  Let $r'=(l_i')$ be another radial which is homotopic to  $r$.
  Then there exists a homotopy $h$ between $r$ and $r'$.
  Let $c=h(0,\cdot)$ be the path between the basepoints of $r$ and $r'$,
  and $\hat c$ the lift of $c$ with $\hat c(0)=\hat x$.
  Then $p(\hat c(1))=r'(0)$.
  We have the lifts $\hat l_i'$ of $l_i'$ with $\hat l_i'(0)=\hat c(1)$,
  and  then $g_i(\hat c(1))=\hat l_i'(1)$ for $i=1,2,\dots,d$.
  Thus we obtain the same  iterated function system
 $(\{g_i\},\hat S)$, and so $\phi_{r'}=\phi_r$.
 \end{proof}

\begin{defi}
 For a geometrically finite rational map $f$,
 we say that $\Omega$ in the proof is an {\em expanding domain} for $f$.
 Note that (i) for any $z\in F$ there exists $n>0$ such that $f^n(z)\notin
 \Omega$, and (ii) for any compact set $K\subset\hat{\mathbb C}-(\mathtt{At}
 \cup\mathtt{Pa})$ there exists $n>0$ such that $f^{-n}(K)\subset \Omega$.
\end{defi}

 \begin{prop}\label{lemcantor}
If $f$ is ss-Cantor, then it is t-Cantor.
\end{prop}

\begin{proof}
  Let $X$ and $X_i$ be path-connected sets as in the definition.
Let us take a basepoint $\bar x\in X$.
We have a radial $r=(l_i)$ with $l_i\subset X$ joining $\bar x$
and $x_i\in f^{-1}(\bar x)\cap X_i$.
Then $\phi_r$ is one-to-one
since $\phi_r(i\cdots)\in X_i$.
\end{proof}

\begin{prop}\label{propgeof}
 Suppose that $f$ is d-Cantor.
 Then $J\cap \mathtt{Cr}=\emptyset$.
 In particular, $f$ is geometrically finite and $P\cap\mathtt{Pa}=\emptyset$.
\end{prop}

\begin{proof}
 Suppose that $J\cap \mathtt{Cr}$ is not empty.
Let  $c\in J$ be a critical value.
Then $\#f^{-1}(c)<d$.
Thus $f|J$ cannot be topologically conjugate to a $d$-to-$1$
map.
\end{proof}

\begin{prop}\label{propsemiconj}
Suppose that $f$ is Cantor.
Then there exists a continuous semi-conjugacy $\phi:\Sigma_d\to J$.
\end{prop}

\begin{proof}
It is easily seen that $P$ has a path-connected compliment.
Thus we can take a radial $r=(l_i)$.

For $\omega=i_1i_2\cdots\in\Sigma_d$, the set of accumulation points
$$\bigcap_{m=1}^\infty
\overline{\bigcup_{n=m}^\infty
L_{i_1i_2\cdots i_{m-1}}(l_{i_mi_{m+1}\cdots i_n})}$$
 is a connected subset of $J$, so it is a singleton.
 Here $L_w(\gamma)$ ($w\in\mathtt{Word}(k)$) is the lift of the curve $\gamma:[0,1]\to \hat{\mathbb C}-P$
 by $f^k$ with $L_w(\gamma)(0)=l_w(1)$.
 Moreover the diameter of
 $\bigcup_{n=m}^\infty L_{i_1i_2\cdots i_{m-1}}(l_{i_mi_{m+1}\cdots i_n})$
converges to zero as $m$ tends to $\infty$.
 Thus the coding $\phi_r(\omega)$ is well-defined and continuous.
\end{proof}

\section{Equivalent conditions}
Suppose that $f$ is a geometrically finite rational map with
$J\cap \mathtt{Cr}=\emptyset$.

\begin{defi}
 Let $X$ be a topological space.
 We say that a continuous mapping $a:X\to\hat{\mathbb C}$ is
 {\em homotopically trivial}
 if  there exists a  homotopy
 $H:X\times[0,1]\to\hat{\mathbb C}- P$ between $a$ and a constant
 mapping.
 In particular, if $a(X)\cap P\ne\emptyset$, then $a$ is
 homotopically nontrivial.

 A subset $X\subset\hat{\mathbb C}$ is said to be homotopically trivial
  if
 the inclusion map is homotopically trivial.
\end{defi}

\begin{remark}\label{rem1}
For $Y\subset X$, if $X$ is homotopically trivial,
then $Y$ is also homotopically trivial.
If a simple closed curve $\gamma$ is homotopically trivial,
then one of the connected components of
 $\hat{\mathbb C}-\gamma$ is  also homotopically trivial.
\end{remark}

\begin{defi}\label{defliftedloop}
 We denote by $M_n:S^1\to S^1$, the map $s\mapsto ns \mod 1$ of
  $S^1=\mathbb R/\mathbb Z$ to itself.
 Let $\gamma:S^1\to\hat{\mathbb C}$ be a closed curve.
 We say a closed curve $\gamma':S^1\to \hat{\mathbb C}$ is a
 {\em lifted loop} of $\gamma$ by $f$ if
  $\gamma\circ M_n  =f\circ\gamma'$  for some $n>0$.
\end{defi}

\begin{defi}\label{defiinvc}
 Let $\Omega$ be an expanding domain.

A closed curve $\gamma:S^1\to \Omega-P$
is {\em homotopically invariant} if $\gamma$ is
homotopically nontrivial,
  there exists a lifted loop $\gamma':S^1\to \Omega-P$ of $\gamma$
  by $f^k$  for some $k>0$, and  there
  exists a homotopy $H:S^1\times[0,1]\to \Omega-P$
  between $\gamma$ and $\gamma'$.
  \end{defi}

  \begin{lemma}\label{lem6}
   Let $p\in\mathtt{At}\cup\mathtt{Pa}$, and let $F_0$ be a connected
   component of the Fatou set $F$ with period $n$
   satisfying $p\in\overline{F_0}$, and
  $U$ a simple domain for $p$ and $F_0$.
   Take $U_k$ as in  Definition \ref{defidomain}.
   Let $\mathcal V_k$ be the set of homotopically nontrivial connected
   components of $\hat{\mathbb C} -U_k$.

  Then for large enough $k$, $\mathcal V_k$ is constant up to homotopy
   relative to $P$.
   Moreover, there is $x\in P$ which is not contained in $F_0$
  	 $\iff$ $\mathcal V_k\ne\emptyset$ for any $k$
     $\iff$ there exists $V\in\mathcal V_k$
  	 such that  $\partial V$ is homotopically invariant.
  \end{lemma}

   \begin{proof}
    Note that each $V\in\mathcal V_k$  is a topological disc
    containing a point in $P$.
   Let $k_0$ be an integer such that $P\cap F_0\subset U_{k_0}$.
   Then $\#\mathcal V_k\le\#\mathcal V_{k+1}$ for $k\ge k_0$.
   Thus there exists $k_1\ge k_0$ such that
   $\#\mathcal V_k$ is constant for  $k\ge k_1$, since
    $\#\mathcal V_k\le \#\{\text{components of $F$
   containing a point of $P$}\}<\infty.$

   Suppose  $k\ge k_1$.
   It is evident that  $P\subset F_0$ if and only if  $\mathcal V_k$ is empty.
  Suppose $\mathcal V_k\ne\emptyset$.
   Then for $V\in \mathcal V_k$, there exists unique $V'\in \mathcal V_{k+1}$
  such that $V'\subset V$ and $(V-V')\cap P=\emptyset$.
   It is easy to see that there exist continuous maps $\gamma:S^1\to
   \partial V$ and $\gamma':S^1\to \partial V'$ which are homotopic in
$\Omega-P$.

   On the other hand,  for $V\in \mathcal V_{k+1}$, there exists
   $V^*\in \mathcal V_k$ such that $\partial V^*=f^n(\partial V)$
   (recall that $n$ is the period of $F_0$).
   Indeed, there exists $V^*$, a connected component of $\hat{\mathbb C}-U_k$
    such that $\partial V^*=f^n(\partial V)$.
   If $V^*\notin\mathcal V_k$, then $V^*$ is homotopically trivial.
   Then  every connected component of $f^{-n}(V^*)$ is also
   homotopically trivial and  simply connected.
   This implies that $V$ is a connected component of $f^{-n}(V^*)$, and
   we arrived at a contradiction.

   Now we consider $V\mapsto V'\mapsto {V'}^{*}$ as a selfmap
   on the finite set $\mathcal V_{k}$.
   It is easy to see that there exists $V\in\mathcal V_{k}$ such that
   $\partial V$ is  homotopically invariant.
   \end{proof}

   \begin{cor}\label{cor2}
    Let $\Omega$ be an expanding domain.
    Then $f^{-n}(\Omega)$ is homotopically nontrivial for any $n$ if and
   	only if there exists a homotopically invariant curve.
   \end{cor}

   \begin{proof}
    We use the notation of Lemma \ref{lem6}.
    Suppose that there exists no homotopically invariant curve.
    By Lemma \ref{lem6}, $P\subset U_k$ for some $k$.
    Then $\hat{\mathbb C}-U_k$ is homotopically trivial, and so
    is $f^{-k}(\Omega)$.

    Conversely, suppose that $f^{-k}(\Omega)$ is homotopically trivial
    for some $k$.
    For any curve $\gamma\subset\Omega$, any connected component of
    $f^{-k}(\gamma)$ is homotopically trivial.
    Thus there exists no homotopically invariant curve.
 \end{proof}

 \begin{lemma}\label{lemmajc}
Let $\Omega$ be an expanding domain, and $\gamma\subset\Omega$
 a homotopically invariant curve.
  Then there exists a closed curve $\bar\gamma\subset J$
  which is a uniform limit of a sequence of curves
homotopic to $\gamma$ in $\Omega-P$.
The curve is nonconstant and
$f^k:\bar\gamma\to f(\bar\gamma)$ is of degree more than one,
where $k$ is the period in the Definition \ref{defiinvc}.
 \end{lemma}

\begin{proof}
 We have a closed curve  $\gamma_1=\gamma':S^1\to \Omega-P$
 and a homotopy  $H_0=H:S^1\times [0,1]\to \Omega$
 as in Definition \ref{defiinvc}.
We inductively obtain
 curves $\gamma_m$ such that $\gamma_{m-1}\circ M_n=f^k\circ \gamma_m$
 and homotopies $H_m:S^1\times [0,1] \to \Omega-P$ ($m=1,2,\dots$)
 between $\gamma_{m-1}$ and $\gamma_m$.
 Since $\gamma_m$ is included in an expanding domain,
 it is easy to see that
$\gamma_m$ uniformly converges to a continuous curve
 in $J$.
 The limit $\bar\gamma$ is nonconstant since it is
 homotopically nontrivial.
Thus $f^k:\bar\gamma\to f(\bar\gamma)$ is of degree more than one.
\end{proof}

For a (super)attracting or parabolic $n$-periodic point $p$,
we use the notaion
$$A(p):=\{z\in F\,:\,\lim_{k\to\infty} f^{kn}(z)=p\}$$
only in this section.

 \begin{defi}
 We say that a fixed point $p$ is {\em solitary} if $p$
 is either (super)attracting or parabolic with the local form
 $z\mapsto z+az^2+\cdots$ around $p$ with $a\ne0$.
 Then the unique Fatou component $A_0(p)$ whose closure contains $p$
 is the {\em immediate  basin of attraction} of $p$.
 \end{defi}

 \begin{defi}\label{defmono}
  Fix a basepoint $\bar x\in\hat{\mathbb C}-P$.
 For a set $X$, the symmetric group of $X$ is
 denoted by $\mathfrak{S}(X)$.
 The {\em monodromy} $\alpha_f:
 \pi_1(\hat{\mathbb C}-P,\bar x)\to\mathfrak{S}
 (f^{-1}(\bar x))$ is an antihomomorphism defined by $\alpha_f([\gamma])(x)
 =y\iff L_x(\gamma)(1)=y$, where $\gamma:[0,1]\to
 \hat{\mathbb C}-P$ is a loop with $\gamma(0)=\gamma(1)=
 \bar x$, and $L_x(\gamma)$ is the lift of $\gamma$ by
 $f$ satisfying
 $L_x(\gamma)(0)=x$.

 Note that $[\gamma]\in\ker\alpha_f \iff$  $L_x(\gamma)$
 is a closed curve for every $x\in f^{-1}(\bar x)$.
 Thus $\ker \alpha_{f^n}\subset\ker \alpha_{f^m}$ if $n>m$.
 Denote $$\alpha_\infty:=\prod_{n>0} \alpha_{f^n}:
 \pi_1(\hat{\mathbb C}-P,\bar x)\to\prod_{n>0}\mathfrak{S}
 (f^{-n}(\bar x)).$$
  Then $\ker\alpha_\infty=\bigcap_{n>0}\ker\alpha_{f^n}$.
  The image of $\alpha_\infty$ is called the {\em iterated
  monodromy group} (see \cite{BarGriNekr03},\cite{Ne05}).

Let $\Omega$ be an expanding domain.
The image of $$\pi_1(\Omega-P,\bar x)\hookrightarrow \pi_1(\hat{\mathbb C}-P,\bar x)
\overset{\alpha_\infty}{\to}\prod_{n>0}\mathfrak{S}
(f^{-n}(\bar x))$$
is called the {\em bounded iterated monodromy group} for $\Omega$.
  \end{defi}

\begin{theorem}\label{th1}
 Let $f$ be a geometrically finite rational map of degree $d$ with
 $J\cap \mathtt{Cr}=\emptyset$.
   Take  an expanding domain   $\Omega$.
 The following are equivalent.
\begin{enumerate}
  \item\label{c} $f$ is Cantor.
\item\label{dc} $f$ is d-Cantor.
 \item\label{wdc} There exist $n>0$  such that $f^n$
      is s-Cantor.
\item\label{f} The Fatou set is connected.
\item\label{jl}
Every connected component of $J$ is simply connected.

 \item\label{jt} $J$ is homotopically trivial in
$\hat{\mathbb C}-P$.
\item\label{cyc} There is no homotopically invariant curve.
 \item\label{cc} There exists  $n>0$ such that for any  closed
      curve $\gamma\subset\Omega$,
any lifted loop of $\gamma$ by $f^{-n}$ is homotopically trivial.
 \item\label{ex}
      There exists $n>0$ such that $f^{-n}(\Omega)$ is homotopically trivial
in $\hat{\mathbb C}-P$.

 \item\label{ab} There exists a solitary fixed point $p$
 such that
 $A_0(p)$ includes $\mathtt{Cr}$.
\item\label{wab} There exists a solitary fixed point $p$
 such that $A_0(p)$
contains at least $2d-4$ critical values counted
with multiplicity, and $A(p)$ includes $P$.
\item\label{c0-5}
     The bounded iterated monodromy group
     $\pi_1(\Omega-P,\bar x)/\ker \alpha_{\infty}$
     is a finite group.
\end{enumerate}
\end{theorem}

\begin{proof}
`(\ref{dc} or  \ref{wdc}) $\Rightarrow$
\ref{c} $\Rightarrow$ \ref{jl}' and `\ref{ex} $\Rightarrow$ \ref{cc} $\Rightarrow$ \ref{cyc}'
 and
 `\ref{f} $\Rightarrow$ \ref{ab} $\Rightarrow$ \ref{wab}' are trivial.

\ref{cyc} $\Rightarrow$ \ref{ex} by Corollary \ref{cor2}.

\ref{cyc} $\Rightarrow$ \ref{ab} by Lemma \ref{lem6}.

\ref{ab} $\Rightarrow$ \ref{ex}.
Using the notation of  Lemma \ref{lem6},
each connected component of $\hat{\mathbb C}-U_k$ is homotopically
 trivial for large $k$.
 By Remark \ref{rem1}, $f^{-k}(\Omega)$ is homotopically trivial.

\ref{wdc} $\Rightarrow$ \ref{jt}.
Let $r$ be a radial for $f^n$.
For $\omega\in\Sigma_{d^n}$, we have a path $l_\omega:[0,1]\to
\hat{\mathbb C}-P$ such that $l_\omega(0)=\bar x,
l_{\omega}(1)=\phi_r(\omega)$.
Thus if $\phi_r$ is homeomorphic, then the map $J
\times[0,1]\to \hat{\mathbb C}-P$ defined by
$(z,t)\mapsto l_{\phi_r^{-1}(z)}(1-t)$
is a homotopy between the inclusion map of $J$ and a constant map.

\ref{ab} $\Rightarrow$ \ref{dc} and \ref{wdc}.
Suppose $P\subset A_0(p)$.
 Let $U$ be a simple domain of $p$, and take
$U_k$ as in Definition \ref{defidomain}.
Then $P\subset U_n$ for some $n$.
The simply connected domain $V=\hat{\mathbb C}-\overline U$
 satisfies $f^{-n}(V)\subset V$ and that any connected component of $V-f^{-n}(V)$
other than $U_n-\overline U$ contains no postcritical point.

 There exists a topological tree
 $\Gamma\subset U_n-U\subset \overline V-f^{-n}(V)$
 satisfying:
\begin{enumerate}
 \item $V':=V-\Gamma$ is connected and simply connected,
 \item $\Gamma$ includes $P\cap V$,
 \item For $k=1,2\dots,n-1$ and every connected component $W$ of $f^{-k}(V)$,
       $W-\Gamma$ is connected.
\end{enumerate}
 Indeed, let $C_i,i=1,2,\dots,m$ be the connected components of
 $\bigcup_{k=0}^{n-1}f^{-k}(\partial V)\cap U_n$, and
$W_j,j=1,2,\dots,s$ be the connected components of $(U_n\cap V)-\bigcup_i C_i$.
Take  points $a_i\in C_i$ and $b_j\in W_j$ for each $i,j$.
 Construct a tree $\Gamma$ with vertex set $\{a_i,b_j\}\cup P\cap V$ and
 edge set $\{e_{ij}:C_i\subset\partial W_j\}\cup\{e_{qj}:q\in P\cap W_j\}$,
 where $e_{ij}$ joins $a_i$ and $b_j$ in $W_j$, and
$e_{qj}$ joins $q$ and $b_j$ in $W_j$.

 Then $f^{-1}(V')$ has exactly $d$ connected components, and so
 the inverse of $f:f^{-1}(V')\to V'$ has $d$ branches.
 Let $g_i:J\to J,i=1,2,\dots,d$ be the branches restricted on the Julia set.

The following fact is crucial.
 For any $w=i_1i_2\cdots i_k\in \mathtt{Word}(k)$,
 some connected component of $f^{-k}(V)$ includes $g_w(J)$,
  where $g_w=g_{i_1}\circ g_{i_2}\circ \cdots\circ g_{i_k}$.
 Indeed, first $J\subset V$, and by induction
 suppose that some connected component $E$ of $f^{-k}(V)$ includes $g_w(J)$.
Then the connected component of $f^{-k}(V)-\Gamma$ included in $E$ includes
$g_w(J)$ by the third condition of $\Gamma$.
Hence for $i\in\mathtt{Word(1)}$, $g_i(g_w(J))$ is included in a connected
component of $f^{-k-1}(V)$.

 Let  $\rho(x,y)$ be the distance  introduced in Lemma
 \ref{lemmetric}.
If $k> n+1$, then $f$ is injective on each connected component of
$f^{-k}(V)$.
Thus the diameters in $\rho$ of connected components of $f^{-k}(V)$
uniformly converge to zero as $k\to\infty$, and consequently
 so do the diameters of $g_w(J),w\in\mathtt{Word}(k)$.
Therefore  there exists
 a conjugacy between $(f,J)$ and $(\sigma,\Sigma_d)$.
Thus \ref{dc} is verified.

It is easier to show \ref{wdc}.
Indeed, take a topological disc $\overline D$ with
$f^{-n}(V')\subset\overline D\subset V'$.

(\ref{jl} or \ref{jt}) $\Rightarrow$ \ref{cyc}.
Assume that there exists a homotopically invariant curve $\gamma$.
 By Lemma \ref{lemmajc}, we have a closed curve $\bar\gamma$
 in $J$ homotopic to $\gamma$, which is not homotopically trivial.
 The connected component of $J$ including $\bar\gamma$ is neither
 simply connected nor homotopically trivial.

 \ref{wab} $\Rightarrow$ \ref{ab}.
Suppose that \ref{wab} is satisfied.
Let $U$ be a simple domain of $p$, and take $U_k$
as in Definition \ref{defidomain}.
By Lemma \ref{lemdomain},
for large $n>0$,  $(P\cup f^{-1}(p))\cap A_0(p)\subset U_n$.
Let $U_{n+1}^0=U_{n+1}$, and let
$U_{n+1}^i,i=1,2,\dots,s$ be the connected components of $f^{-1}(U_n)$
other than $U_{n+1}$.
Note that $U_{n+1}^i\not\subset A_0(p)$ for $i=1,2,\dots,s$
since $f^{-1}(p)\cap A_0(p)
\subset U_n$.

Let $W=U_n-\bigcup_{a:\text{ critical value}} B_a$,
and ${W^i}=U_{n+1}^i-\bigcup_a f^{-1}(B_a)$
the connected component of $f^{-1}(W)$ included in
$U_{n+1}^i$, where $B_a$ is a small disc around $a$.
Let $m_i$ be the number of critical points in $W^i$
counted with multiplicity, and $d_i=\mathrm{deg}(f:W^i\to W)$.
Note $\sum_{i=0}^sd_i=d$.

We show
\begin{equation}\label{eq:eu}
\sum_{\gamma:\text{ connected component of }
\partial U_{n+1}^i}
(\mathrm{deg}(f:\gamma\to f(\gamma))-1)=2d_i-2-m_i
\end{equation}
for $i=0,1,\dots,s$.
To this end, first note that
$\chi(W^i)=d_i\chi(W)$, where $\chi$ denotes the Euler
characteristic.
We have
$$\begin{array}{cl}
\#\{\text{connected components of } \partial W^i\}&
=-\chi(W^i)+2=-d_i\chi(W)+2,\\
\#\{\text{connected components of } \partial W\}&
=-\chi(W)+2.
\end{array}
$$
This implies that
\begin{multline*}
\sum_{a\in \mathtt{Cr}\cap U_{n+1}^i}(\mathrm{deg}(f:\partial B'_a\to
\partial B_{f(a)})-1)\\+
\sum_{\gamma:\text{ connected component of }
\partial U_{n+1}^i}
(\mathrm{deg}(f:\gamma\to f(\gamma))-1)
\\
=
\sum_{\gamma:\text{ connected component of } \partial W^i}
(\mathrm{deg}(f:\gamma\to f(\gamma))-1)\\
=d_i(-\chi(W)+2)-(-d_i\chi(W)+2)=2d_i-2,
\end{multline*}
where $B_a'$ is a connected component of $f^{-1}(B_{f(a)})$
including $a$.
Note that
$\sum_{a\in \mathtt{Cr}\cap U_{n+1}^i} (\mathrm{deg}(f:\partial B'_a\to
\partial B_{f(a)})-1)=m_i$.
Hence (\ref{eq:eu}) holds.

We have supposed that $U_n$ has $2d-4$ critical values counted with
multiplicity, in other words, $\sum_{i=0}^sm_i\ge 2d-4$.
From (\ref{eq:eu}), we have
$$0\le \sum_{i=0}^s(2d_i-2-m_i)\le2d-2(s+1)-(2d-4)
=2-2s,$$ so we have $s=0$ or $1$.

If $s=0$, then $f^{-1}(A_0(p))=A_0(p)$, and so
$A(p)=A_0(p)$, and every critical point is  contained in $A_0(p)$.

Suppose $s=1$.
It follows from the above inequality that
 $2d_i-2-m_i=0\ (i=0,1)$, and thus
 \begin{equation}\label{eq:m0m1}
   m_0+m_1=2d-4.
   \end{equation}
We have
$$0=2d_i-2-m_i=\sum_{\gamma:\text{ connected component of }
\partial U_{n+1}^i} (\mathrm{deg}(f:\gamma\to f(\gamma))-1)$$ for $i=0,1$.
This means that for any connected component $\gamma\subset \partial U_n$ and
any connected component $\gamma'\subset f^{-1}(\gamma)$, the degree
of $f:\gamma'\to\gamma$ is one.
However, the equality (\ref{eq:m0m1}) together with
the fact that the whole sphere has $2d-2$ critical values counted with
multiplicity implies that $A(p)-A_0(p)$ contains a
critical value.
We have an invariant curve $\tilde\gamma$ which
is a  connected component of $\partial U_N$ for large $N>0$
by Lemma \ref{lem6}.
Moreover, $\deg(f^k:\tilde\gamma'\to\tilde\gamma)>1$ by Lemma \ref{lemmajc}.
This is a contradiction, and so $s=0$.


\ref{ex} $\Rightarrow$ \ref{c0-5}.
 Suppose that every connected component of
 $f^{-n}(\Omega)$ is homotopically trivial.
 To see  $\alpha_\infty(\pi_1(\Omega-P,\bar x))\subset
 \prod_{k=0}^\infty\mathfrak S(f^{-k}(\bar x))$ is a finite group,
 we show that the projection
 $\alpha_\infty(\pi_1(\Omega-P,\bar x))\to
 \mathfrak S(f^{-n}(\bar x))$ is injective.
 To this end, let $\gamma:[0,1]\to\Omega-P$ be a closed curve such that
 $g=[\gamma]\in \pi_1(\Omega-P,\bar x)$ satisfies $\alpha_{n}(g)=1$.
 We denote by $L_{x,k}(\gamma)$ the lift of $\gamma$ by $f^k$ with
 $L_{x,k}(\gamma)(0)=x$.
 Then $L_{x,n}(\gamma)$ is a closed curve for any $x\in f^{-n}(\bar x)$.
 This implies that $L_{x,k}(\gamma)$ is a closed curve for
 any $0\le k\le n-1$ and any $x\in f^{-k}(\bar x)$.
 Moreover, for any $x\in f^{-n}(\bar x)$, the
 closed curve $L_{x,n}(\gamma)$ is homotopically trivial, since
 it lies in $f^{-n}(\Omega)$.
 Thus every $L_{x,k}(\gamma), x\in f^{-k}(\bar x),k>0$ is a closed curve.
 Consequently, $\alpha_\infty(\pi_1(\Omega-P,\bar x))\to
 \mathfrak S(f^{-n}(\bar x))$ is injective, and so
 $|\alpha_\infty(\pi_1(\Omega-P,\bar x))|\leq|\mathfrak S(f^{-n}(\bar x))|$.

 \ref{c0-5} $\Rightarrow$  \ref{cyc}.
Suppose that there exists an invariant curve $\gamma$.
 Let $\gamma'$ and $k$ be as in Definition \ref{defiinvc}.
 Then $\deg(f^k:\gamma'\to\gamma)>1$.
 Thus for any $m>0$, there exist $n>0$ and $x\in f^{-n}(\bar x)$
 such that $L_{n,x}(\gamma^m)$ is not a closed curve.
 Therefore for the element $g\in\pi_1(K-P,\bar x)$ corresponding
 to $\gamma$, $\alpha_\infty(g)$ is not of finite order.
\end{proof}

\begin{remark}
 The number $2d-4$ in \ref{wab} of Theorem \ref{th1} is
 the best possible.
 Namely, for $d\ge3$ there exists a rational map of degree $d$
 with solitary fixed point $p$ such that $A(p)$ contains
 all critical points, and such that $A_0(p)$ contains exactly $2d-5$
 critical points  counted with multiplicity.
\end{remark}

\begin{remark}

      In Definition \ref{defiinvc}, the condition $\gamma\subset\Omega-P$
      is important.
      The following examples explain the difference between $\Omega-P$
and  $\Omega-\overline P$, $\hat{\mathbb C}-\overline P$, or $\hat{\mathbb C}-P$.
\begin{enumerate}
 \item[(a)]
      There exists a geometrically finite rational map
      that does not satisfy the condition of Theorem \ref{th1}
      and has no homotopically invariant curve in
      $\Omega-\overline P\subset \hat{\mathbb C}-\overline P$.

 \item[(b)]
	 There exists a hyperbolic rational map
      that  satisfies the condition of Theorem \ref{th1}
      and has a homotopically invariant curve in
      $\hat{\mathbb C}-P$.
\end{enumerate}

 Set $f_\lambda(z)=\lambda z/(z^2+1)$ for $\lambda\in\mathbb R$.
      We have    $\mathtt{Cr}=\{\pm1\}$ and $P\subset \mathbb R-\{0\}$.
      It is easy to see that $0$ is a fixed point with multiplier $\lambda$.
      The curve $\gamma=\{ia:a\in\mathbb R\}\cup\{\infty\}$ is a completely invariant set for $f_\lambda$.
      If $\lambda=\pm1$, then $f_\lambda$ does not satisfy
      the condition of Theorem   \ref{th1}, and $0\in \Omega-P$ but
      $0\notin \Omega-\overline P$
      (the case (a)).
    If $-1<\lambda<1$, then $f_\lambda$ satisfies the condition of Theorem
\ref{th1}, and $0\notin \Omega-P$ but $0\in \hat{\mathbb C}-P$  (the case (b)).

\end{remark}

\section{Strongly Cantor rational maps}
In this section, we give a sufficient condition for $f$ to be
s-Cantor.

\begin{theorem}\label{th:8}
 Let $f$ be a Cantor rational map with a (super)attracting fixed point $\infty$.
 Suppose that the critical orbits $\mathtt{Orb}(f(c)),c\in\mathtt{Cr}$ are
 distjoit (i.e. $\mathtt{Orb}(f(c))\cap \mathtt{Orb}(f(c'))=\emptyset$ for distinct $c,c'\in\mathtt{Cr}$), and that  $\infty$ is outside other critical orbits
  (i.e. $\infty\notin \mathtt{Orb}(f(c))$ for $c\in\mathtt{Cr}_0:=\mathtt{Cr}-\{\infty\}$).
 Then $f$ is s-Cantor.
\end{theorem}

\begin{proof}
 For a connected open set $W$ containing $\infty$ and an integer $k\ge0$,
  let us denote
 by $W(k)$ the connected component of $f^{-k}(W)$ containing $\infty$.
Note that $W(i)(j)=W(i+j)$ for $i,j\ge0$.

 We say a connected open set $W$ containing $\infty$
 is {\em appropriate} if $\partial W$
 is a simple closed curve,
 $\overline W\subset f^{-1}(W)$ (i.e. $\overline{f(W)}\subset W$), and
  for any $c\in\mathtt{Cr}_0$,
  $(W(1)-\overline{W(0)})\cap \mathtt{Orb}(f(c))\ne\emptyset$.
  We define $\kappa=\kappa(W)$ to be the maximal integer such that
 for any $c\in\mathtt{Cr}_0$,
 $(W(\kappa)-\overline{W(\kappa-1)})\cap \mathtt{Orb}(f(c))\ne\emptyset$.
 We define $\lambda=\lambda(W)$ to be the maximal integer such that
 $W(\lambda)\cap\mathtt{Cr}_0=\emptyset$.
 Clearly, $\kappa\leq\lambda$.
If $W$ is appropriate, then $W(i)$ for $1\le i\le\lambda$ is a topological disc.

 In the sequel, we will construct an appropriate domain $W$ such that
 $\kappa=\lambda$ and
 $\mathtt{Cr}_0\subset W(\kappa+1)-\overline{W(\kappa)}$.
 To this end,  we take an initial appropriate domain $W_0$ near $\infty$
 and  deform it in several steps.
 If $\infty$ is a superattracting fixed point of degree $\delta$,
 then take  a B\"ottcher coordinate
 $\varphi:U\to\mathbb D=\{|z|<1\}$ for  the superattracting fixed point $\infty$.
Here $U$ is a connected and simply connected neighborhood of $\infty$.
 If $0<r<1$, for any $c\in\mathtt{Cr}_0$,
 $\varphi^{-1}(\{0<|z|\le r\})$ contains $f^m(c)$ for some $m>0$,
 since we have assumed that the orbit of $c$ does not land at $\infty$.
 Let $r$ satisfy $\varphi^{-1}(\{|z|= r\})\cap P=\emptyset$.
 For a large enough integer $N$, we define
 $W_0=\varphi^{-1}(\{|z|<r^{\delta^N}\})$.
 Then $W_0$ is appropriate with $\kappa(W_0)\ge N$.
 If $\infty$ is attracting,
 then take a linealizing coordinate $\varphi:U\to\mathbb C$
 for $\infty$,
 and define $W_0$ similarly such that $\kappa(W_0)=N$ is large enough.
The integer $N$ will be determined later.

 Let $W$ be an  appropriate domain.
 We use the notation $A_{i}=A_W(i)=W(i)-\overline{W(i-1)}$
 for $i=1,\dots,\lambda$.
 Note that $A_i$ is an annulus and is the connected component
 of $f^{-1}(A_{i-1})$ such that  $A_{i-1}$ and $A_{i}$ share a
 boundary component.
 The restriction $f:A_{i}\to A_{i-1}$ is a covering of degree $\delta$.
 We set $A_{\lambda+1}$ to be the connected component of $f^{-1}(A_\lambda)$
 such that  $A_{\lambda}$ and $A_{\lambda+1}$ share a boundary component.
 Note that $A_{\lambda+1}\subset W({\lambda+1})-\overline{W({\lambda})}$
 is not an annulus.
 Let us denote $C_i=C_W(i)=\partial W(i)$  for $i=0,1,\dots,\lambda$.
 It equals $\overline{A_{i+1}}\cap\overline{A_i}$ for $i=1,2,\dots,\lambda$.

 Let $\mathcal B_m$ be the set of connected components of $f^{-m}(A_1)$,
 and
 \begin{equation}\label{eq:0}
  \mathcal B=\mathcal B_W=\bigcup_{m=0}^\infty \mathcal B_m.
   \end{equation}
Note that $A_i\in \mathcal B_{i-1}$ for $i=1,2,\dots,\lambda+1$.
 We say distinct $X,X'\in\mathcal B$ are {\em adjacent} if they share
a boundary component.
 If $X,X'\in\mathcal B-\{A_1\}$ are adjacent,
 then $f(X),f(X')\in\mathcal B$ are also adjacent.
 Define an order $\prec$ on $\mathcal B$ by
  $X\prec X'$ if $X'$ separates $X$ and $W$.
  We may use the notation $X\prec_W X'$ by specifying the domain.
 It is easily seen that for any $X\in\mathcal B-\{A_1\}$, there uniquely
 exists $X'\in\mathcal B$ such that $X,X'$ are adjacent and $X\prec X'$.
 We use the notation  $X+1$ for this $X'$,
 and define inductively $X+m=(X+m-1)+1$ for
 $m=2,3,\dots$.
 If $X'=X+m$, we write $$L(X,X'):=m.$$
If $X,X'$ are adjacent, then $X=X'+1$ or $X+1=X'$.
For $X\in\mathcal B-\{A_1\}$, the set $\{Y\in\mathcal B:X\prec Y\}$
is linearly ordered.

\medskip

Claim 1.
 If $X\prec X'$, then there exists $m>0$
 such that $X+m=X'$.

{\em Proof of Claim 1.}
 There is a path $\gamma$ in the Fatou set $F$ which connect
 $X$ and $X'$.
 By the compactness of $\gamma$, we have
 $\gamma\subset \bigcup_{0\le k \le k_0} f^{-k}(W)$ for some $k_0$.
 We have
$F-\bigcup_{X\in \mathcal B}\overline X=W\cup\bigcup_{m=0}^\infty
 f^{-m}(f^{-1}(W)-W(1))$.
 The connected components of $f^{-1}(W)$ other than $W(1)$ are topological
 disc with no point of $P$.
 Thus all the connected components of $F-\bigcup_{X\in \mathcal B}\overline X$
  are topological discs.
 Since only finite of them intersect $\gamma$, we can retake $\gamma$ such that
 $\gamma\subset \bigcup_{0\le k \le k_0-1} f^{-k}(A_1)$.
 Hence $\#\{X\in\mathcal B:X\cap\gamma\ne\emptyset\}<\infty$, and
 the claim is true.

\medskip

 Claim 2.
 Let $k>0$ be an integer.
 Suppose $f^k(X')= X$.
 Then
 \begin{itemize}
   \item $X\ne X'$
   \item $X\not\prec X'$
   \item $X\not\prec X'+1$
 \end{itemize}


 {\em Proof of Claim 2.}
The first assertion is easy.

Assume $X\prec X'$.
 Let $X_0=X$ and $X_i=X_0+i$ for $i=1,2,\dots,s,\dots,m$, where $X_s=X'$ and $X_m=A_1$.
 We have  $\{f^k(X_j):0\le j\le m-1\}\cap \mathcal B\supset
 \{f^k(X_s),f^k(X_s)+1,f^k(X_s)+2,\dots,A_1\}$, since $f^k(X_j)$ and
 $f^k(X_{j+1})$ are adjacent.
 Hence $L(f^k(X_s),A_1)\le m$.
 This is a contradiction.

 Assume $X\prec X'+1$.
 Then $X+1=X'+1$ or $X+1\prec X'+1$.
 Since $f^k(X')=X$, $f^k(X'+1)$ is adjacent to $X$.
 If $X+1=f^k(X'+1)$, then $f^k(X'+1)=X'+1$ or $f^k(X'+1)\prec X'+1$,
 which contradicts the first or second assertion.
 If $f^k(X'+1)+1=X$, then $f^k(X'+1)\prec f^k(X'+1)+1=X\prec X'+1$,
 which contradicts to the second assertion.
This completes the proof of Claim 2.

For $z\in\bigcup_{X\in\mathcal B} X$,
 let $X(z)=X_W(z)\in\mathcal B$ be the connected component containing $z$.





 Let $c$ be a critical point other than $\infty$.
 Suppose $$c\notin A_{\lambda+1},$$
 equivalently, $X(c)\prec A_{\lambda+1}$.
 Define an integer
 $\mu=\mu(c)$ satisfying $f^\mu(c)\in A_1$.
 Let $\gamma:[0,1]\to \overline{A_1}$ be a simple path such that
 $\gamma(0)=f^\mu(c)$, $\gamma(1)\in \partial A_1$,
 and $\gamma(t)\in A_1-P$ for $0<t<1$.
 We say $\gamma$ is of {\em V-type} for $(c,W)$ if $\gamma(1)\in C_0=\partial W$,
 and is of {\em $\Lambda$-type} for $(c,W)$ if  $\gamma(1)\in C_1=\partial W(1)$.

 The following lemma will be proved later.

\begin{lemma}\label{lem:10}
 Suppose that $X(c)$ and $X\in\mathcal B$ are adjacent.
 Then there exists a simple path $\gamma$ of V-type or $\Lambda$-type
 for $(c,W)$ such that the connected component $\gamma'$ of $f^{-\mu}(\gamma)$
 containing $c$ has an endpoint in $E_0=\overline{X(c)}\cap \overline{X}$.

\end{lemma}

By this lemma, we have  a simple path $\gamma$ of V-type or $\Lambda$-type
 such that the connected component $\gamma'$ of $f^{-\mu}(\gamma)$
 containing $c$ has an endpoint in $\overline{X(c)}\cap \overline{X(c)+1}$.
Let  $\gamma_j:[0,1]\to f^{-j}(\overline{A_1})$ for $1\le j\le \mu-1$
 be the lift of $\gamma$ by $f^{j}$
 with  \begin{equation}\label{eq:0.3}\gamma_j(0)=f^{\mu-j}(c).
       \end{equation}
Note that $\gamma_j(0)=f^{\mu-j}(c)\in A_{j+1}$ for $1\le j\le \kappa-1$.

Set $$\sigma=\left\{\begin{array}{ll}
	      0& 	     \text{ (V-type) }\\
		     1&\text{ ($\Lambda$-type).}
  \end{array}\right.$$

 We take a Jordan domain $U_\gamma\subset A_1$ including $\gamma\cap A_1$
 which is so close to $\gamma$ that $U_\gamma$ contains
 no point of $P$ other than $\gamma(0)=f^{\mu}(c)$, and that
 $\partial U_\gamma\cap \partial A_1$ is an arc in  $C_\sigma$.
Define $W_\gamma:=\mathrm{int}(\overline{W\cup U_\gamma})$ (V-type) or
$W_\gamma:=W(1)-\overline{U_\gamma}$ ($\Lambda$-type).

\medskip

  Claim 3.
Suppose  $\kappa(W)\ge2$.
 The domain $W_\gamma$ is an appropriate domain with
 $\kappa(W)-1\le\kappa(W_\gamma)\le\kappa(W)$ and $\lambda(W_\gamma)=\lambda(W)-\sigma$.
 Let
 $$\mathcal B^0=\mathcal B_{c,W}^0=
 \{X\in \mathcal B_W: X_W(c)+1 \prec_W X\preceq_W A_W(\sigma+1)\}, $$
  and
 $$\mathcal B^1=\mathcal B_{c,W_\gamma}^1
 =\{X\in \mathcal B_{W_\gamma}: X_{W_\gamma}(c) \prec_{W_\gamma}X \}.$$
 Then we have a one-to-one correspondence
 $\Psi=\Psi_\gamma:\mathcal B^0\to \mathcal B^1$ such that
 $X$ and  $\Psi(X)$ are homotopic relative to $P-\mathtt{Orb}(f(c))$,
 and $\Psi$ preserves the order $\prec$ and  the $+1$ operation.
In particular, $\Psi(A_W(\lambda(W)))
=A_{W_\gamma}(\lambda(W)-\sigma)=A_{W_\gamma}(\lambda(W_\gamma))\in \mathcal B^1$.

 {\em Proof of Claim 3.}
We have $W\subset W_\gamma\subset\overline{W_\gamma}\subset W(1)$
 (V-type) or
 $\overline W\subset W_\gamma\subset W(1)$
 ($\Lambda$-type).
 Hence $\overline{f(W_\gamma)}\subset W_\gamma$.
 From $W_\gamma\cap P\subset W(1)\cap P$ and $\kappa\ge2$,
we can see that $W_\gamma$ is appropriate.

By Claim 2, if $z\in f^{-k}(c)$ for some $k>0$, then $X_W(z),X_W(z)+1\notin \mathcal B^0$.
Therefore if $X\in\mathcal B^0$, then any adjacent $X'\in\mathcal B$ to $X$
 contains no preimage of $c$.
The same statement is true for $\mathcal B^1$.

Take a point $x\in A_W(\sigma+1)\cap A_{W_\gamma}(1)$.
We define $\Psi(A_W(\sigma+1))=A_{W_\gamma}(1)$, and inductively
$\Psi(X)$ is the connected component of $f^{-1}(\Psi(f(X)))$
 such that $X$ and  $\Psi(X)$ are homotopic relative to
 $\bigcup_{k=0}^{\infty} f^{-k}(x)$.
This procedure is valid for $X\in\mathcal B^0$ because
every connected component $U$ of $\bigcup_{k=0}^{\infty} f^{-k}(U_\gamma)$
adjacent to $X$
is a topological disc such that $\partial U\cap\partial X$ is connected
and $U$ is not adjacent to $X'\in\mathcal B$ other  than $X$.
Here we say $U$ and $X$ are adjacent if $U\cap X=\emptyset$ and
$\partial U\cap\partial X\ne\emptyset$.
The symmetric difference between $X$ and $\Psi(X)$ is a disjoint union of such
topological discs.

Let $U_c$ be the connected component of $f^{-\mu}(U_\gamma)$
containing $c$.
Then $U_c$ is included in $X_W(c)$ and adjacent to $X_W(c)+1$
by the definition of $\gamma$.
Thus $$X_{W_\gamma}(c)= (X_W(c)+1)\cup U_c\cup \tilde X\cup U_+-U_-,$$
where $\tilde X$ is the union of $X\in\mathcal B_W$ adjacent to $U_c$,
 and each of $U_-,U_+$ is a disjoint union of  topological discs.
(Other $X$'s in $\mathcal B_W$ do not contribute to $X_{W_\gamma}(c)$.
Indeed,  if $X\in\mathcal B_W$ is adjacent to $U_c$, then
any $X'\in\mathcal B_W$ adjacent to $X$ does not fulfill $f^k(X')=X_W(c)$
for  $k>0$ by Claim 3, because of the fact that $f^k(X)=f^k(X_W(c)+1)$
and that $f^k(X')=X_W(c)\Rightarrow f^k(X)\text{ is adjacent to }X_W(c)$.)
 Therefore
 we have $\Psi(X_W(c)+2)=X_{W_\gamma}(c)+1$.

 Hence $\Psi:\mathcal B^0\to\mathcal B^1$ is well-defined and  $X\cap\mathtt{Orb}(c')=\Psi(X)\cap\mathtt{Orb}(c')$ for
 every $X\in\mathcal B^0$ and  $c'\in\mathtt{Cr}-\{c,\infty\}$.
 In particular,
 $A_{W_\gamma}(i)=\Psi(A_W(i+\sigma))$
 for $1\le i\le \lambda(W)-\sigma$.
Clearly,
$$A_{W_\gamma}(i)\cap(P\cup\mathtt{Cr}-\mathtt{Orb}(c))
=A_W(i+\sigma)\cap(P\cup\mathtt{Cr}-\mathtt{Orb}(c))$$
and
$$A_{W_\gamma}(i)\cap\mathtt{Orb}(c)=A_W(i+1-\sigma)\cap\mathtt{Orb}(c)$$
for $1\leq i\leq \lambda(W)-\sigma$.
Thus   $\kappa(W_\gamma)\geq\kappa(W)-1$ and
 $\lambda(W_\gamma)\geq\lambda(W)-\sigma$.
 Since $A_{W_\gamma}(\lambda(W)+1-\sigma)\cap\mathtt{Cr}\supset
A_W(\lambda(W)+1)\cap\mathtt{Cr}$, we have
 $\lambda(W_\gamma)=\lambda(W)-\sigma$.

Case 1. Suppose $\gamma$ is V-type.
If $A_W(\kappa(W)+1)\cap\mathtt{Orb}(c)\ne\emptyset$,
then $\kappa(W_\gamma)=\kappa(W)$; otherwise $\kappa(W_\gamma)=\kappa(W)-1$.

Case 2-i. Suppose $\kappa(W)<\lambda(W)$ and $\gamma$ is $\Lambda$-type.
If $A_W(\kappa(W)+1)\cap\mathtt{Orb}(c')\ne\emptyset$ for every ${c'\in\mathtt{Cr}-\{c,\infty\}}$,
then $\kappa(W_\gamma)=\kappa(W)$; otherwise $\kappa(W_\gamma)=\kappa(W)-1$.

Case 2-ii. Suppose $\kappa(W)=\lambda(W)$ and $\gamma$ is $\Lambda$-type.
Then $\kappa(W_\gamma)=\kappa(W)-1$.

It is clear that $\Psi$ preserves the order $\prec$ and the $+1$ operation.
This completes the proof of Claim 3.
We call this operation a $\gamma$-deformation of an appropriate domain.

\medskip

 In the proof above, we have proved $X_{W_\gamma}(c)+1=\Psi(X_W(c)+2)$.
 This means that
 $L(X_{W_\gamma}(c),A_{W_\gamma}(\lambda(W_\gamma)))
 =L(X_W(c),A_W(\lambda(W)))-1$.

 We show that if $\kappa(W)\ge {L}(X_W(c),A_W(\lambda(W)))
 =:n$, then
by repeating $\gamma$-deformation
 $n-1$ times, we get an appropriate domain $W'$ such that
 $c\in A_{W'}(\lambda(W')+1)$ and
 $A_W(\lambda(W)+1)\cap (P\cup\mathtt{Cr}-\mathtt{Orb}(c))
 = A_{W'}(\lambda(W')+1)\cap (P\cup\mathtt{Cr}-\mathtt{Orb}(c))$.
Indeed, we take a sequence of paths $\gamma_1,\gamma_2,\dots,\gamma_{n-1}$
 which satisfies the following:
 $\gamma_k$ is of V-type or $\Lambda$-type for $(c,W^{k-1})$
  obtained by Lemma \ref{lem:10},  where
 inductively we write $W^0=W$ and $W^k=W^{k-1}_{\gamma_k}$
 ($k=1,2,\dots,n-1$).
 Then $X_{W^{n-1}}(c)+1=\Psi_{\gamma_{n-1}}\circ\cdots\circ
 \Psi_{\gamma_1}(A_W(\lambda(W)))=:\tilde X$.
We have $\tilde X=A_{W^{n-1}}(\lambda(W^{n-1}))$, since
$\Psi_{\gamma_k}\circ\cdots\circ
 \Psi_{\gamma_1}(A_W(\lambda(W)))\in\mathcal B_{W^k}^0$
 for $k=1,2,\dots,n-2$.

Taking $N$ large enough, we can apply this procedure to all the critical points
other than $\infty$, and
 we have an appropriate domain $\tilde W$ such that $\tilde W(\kappa)=\tilde W(\lambda)$
contains all the critical values.
\end{proof}

 {\em Proof of Lemma \ref{lem:10}.}
 We assume $f^\mu(X)=f(A_1)$ without loss of generality.
 Then $E_0$ is a connected component of $f^{-\mu}(C_0)$.
 Let $P_0\subset A_1$ be the set of critical values of $f^\mu:X(c)\to A_1$.
 Let $m=|P_0|$ and take $m$ points $a_0,a_1,\dots,a_{m-1}\in C_0$
 arranged counterclockwise.
 Take $\mathcal L=\{L_p\}_{p\in P_0}$, a collection of
 disjoint simple paths $L_p$'s in $A_1$ joining $p\in P_0$ and some $a_{j(p)}$
 such that $j(p)\ne j(p')$
 if $p\ne p'$, and such that $L_p$ does not intersect  $P$ except at the endpoint.
 We consider that  $j:P_0\to\mathbb Z/(m)$ is a bijection which depends on
 $\mathcal L$.
 We may write the inverse $p=j^{-1}:\mathbb Z/(m)\to P_0$.
  Write $L=\bigcup_{p\in P_0}L_p$.
 Since each connected component of $f^{-\mu}(A_1-L)$
 in $X(c)$ is an annulus,
 $f^{-\mu}(C_0\cup L)\cap \overline{X(c)}$ is connected.


 For $z\in \overline{X(c)}$,
 let us denote by $I(z)$ the  connected component of
 $f^{-\mu}(L)$ containing $z$ if it exists, and similarly by $\overset{\circ}I(z)$
 the  connected component of $f^{-\mu}(L-P_0)$ containing $z$.
  We say $\mathcal C=(\{E_i\}_{i=0}^k,\{c_i\}_{i=0}^k)$ is a chain between $E_0$ and $c$ for
 $\mathcal L$ if
  $E_0,E_1,\dots,E_k$ are distinct connected  components of
 $f^{-\mu}(C_0)$ in $\overline{X(c)}$ and   $c_0,c_1,\dots,c_{k}$
 are distinct critical points of $f^\mu|_{X(c)}$ satisfying the following:
 $I_i=I(c_i)$ joins  $E_i$ and $E_{i+1}$ ($i=0,1,\dots,k-1$),
 $I_k=I(c_k)$ has an endpoint in $E_k$,
and $c_k=c$.
The existence of such a chain is guaranteed by the fact mentioned at the end of
 the previous paragraph.
 We say $k$ is the length of the chain, and denote by
 $\mathrm{Len}(\mathcal C)$.
 If there exists a chain with $\mathrm{Len}(\mathcal C)=0$,
 then $L_{f^\mu(c)}$ is a required path $\gamma$.

If $m=1$, then $c$ is the only critical point of $f^{\mu}:{X(c)}\to A_1$.
In this case,  any chain between $E_0$ and $c$ has  length zero.
Therefore we can assume  $m>1$.

 Take a chain $\mathcal C=(\{E_i\},\{c_i\})$ between $E_0$ and $c$
 for $\mathcal L$  minimizing $k=\mathrm{Len}(\mathcal C)$.
 Moreover, if there are several candidates for $E_k$, then choose one
 minimizing $\mathrm{D}(\mathcal C)$ defined below.
 Assume $k>0$.
 We will show that we can deform $\mathcal L$ such that there exists
 a chain of length $0$.
 To this end, we use the deformation by half Dehn twists several times.
 At each step, either the length $\mathrm{Len}(\mathcal C)$ of the chain or
 the ``distance'' $\mathrm{D}(\mathcal C)$
 between $I_{k}$ and $I_{k-1}$ reduces by at least one.

Let $p_0=f^\mu(c)$ and $p_1=p(j(p_0)+1)$.
For simplicity, we set $j(p_0)=0,j(p_1)=1$.
Let us denote by $[a_0,a_{1}]$  the counterclockwise
 arc in $C_0$ from $a_0$ to $a_{1}$.
 Take a simple path $l$ between $p_0$ and $p_1$ in $A_1$
 which does not intersect $L$ except at
 the endpoints such that the closed curve
 $L_{p_0}\cup l\cup L_{p_1}\cup [a_{0},a_{1}]$ is homotopically
 trivial in $\overline{A_1}$ (i.e. it does not enclose $A_2$).
Note that $l$ is unique up to homotopy.
 Take a Jordan domain $U_l\subset (A_1-L)\cup (L_{p_0}\cup L_{p_1})$
 which includes $l$ and
 does not include any  point of $P$ other than $p_0,p_1$.

 We define a half Dehn twist on $U_l$ as follows.
 Let $\tau:\mathbb D\to \mathbb D$ be the homeomorphism
 defined by $\tau(z)=-z$ for $|z|\le1/2$ and $\tau(z)=e^{2(1-|z|)\pi i}z$
 for $1/2<|z|\le1$.
 Note that $\tau([1/2,1))$ is a path joining $-1/2$ and $1$,
  and $\tau((-1,-1/2])$ is a path joining $-1$ and $1/2$.
 Choose an orientation-preserving diffeomorphism $h:\overline{U_l}-l\to\overline{\mathbb D}-\{0\}$
  with
$h(L_{p_0}\cap \overline{U_l}-\{p_0\})=(0,1]$, and
$h(L_{p_1}\cap \overline{U_l}-\{p_1\})=[-1,0)$.
Let $T=h^{-1}\circ \tau\circ h:\overline{U_l}-l\to \overline{U_l}-l$.
 Extend $T$  to a self-homeomorphism of $A_1-l$
 by the identity outside $U_l$.


 We take a counterclockwise universal covering $\zeta:\mathbb R\to C_0$ such that
 $\zeta(i+n/m)=a_{n}$
 for $i,n\in\mathbb Z$.
 For a connected component $E$ of $f^{-\mu}(C_0)$, and
 $\tilde a_0\in f^{-\mu}(a_0)\cap E,
 \tilde a\in f^{-\mu}(\{a_0,a_1,\dots,a_{m-1}\})\cap E$,
  we define $i=d(\tilde a_0,\tilde a)$ to be the smallest positive integer
  such that $\tilde\zeta_{\tilde a_0}(i/m)=\tilde a$,
  where $\tilde\zeta_{\tilde a_0}:
  \mathbb R\to E$
  is the lift of $\zeta$ by $f^\mu$ with $\tilde\zeta(0)=\tilde a_0$.

  Let
  $$\mathrm{D}(\mathcal C)=\mathrm{D}_{\mathcal L}(\mathcal C)=
  \min\{d(\tilde a_0,\tilde a):\tilde a_0\in E_k\cap I_k,\tilde a\in E_k\cap I_{k-1}\}.$$
  We can consider $\mathrm{D}(\mathcal C)$ to be a counterclockwise
  distance from $I_{k}$ to $I_{k-1}$  along $E_k$.

The following lemma is easy to see.

\begin{lemma}\label{lem:10-0}
  Let $E$ be a connected component of $f^{-\mu}(C_0)$,
let $\tilde a_0\in f^{-\mu}(a_0)\cap E,\tilde a_1\in f^{-\mu}(a_1)\cap E$
such that $d(\tilde a_0,\tilde a_1)=1$.
Let $\alpha:\mathbb R\times(0,1]\to\overline{U_l}$
be the composition of the universal covering
$\mathbb R\times(0,1]\to \overline{\mathbb D}-\{0\}$ and
$h^{-1}:\overline{\mathbb D}-\{0\}\to\overline{U_l}-l$
such that $\alpha(n,(0,1])\subset L_{p_0}$,
$\alpha(n+1/2,(0,1])\subset L_{p_1}$ for $n\in\mathbb Z$.
Let $\tilde\alpha_{\tilde a_0}:\mathbb R\times(0,1]\to S$ be the lift
of $\alpha$ by $f^\mu$ (i.e. $f^\mu\circ\tilde\alpha_{\tilde a_0}=\alpha$)
 such that
$\tilde\alpha_{\tilde a_0}(0,(0,1])\subset \overset{\circ}I(\tilde a_0)$.


For $\mathcal{L}'=\{T^n(L_p)\}_{p\in P_0}$ with $n\in\mathbb Z$,
we use the notation $I^n(z)$ defined similarly to $I(z)$.
Then the following statements hold:
 \begin{itemize}
   \item $\tilde\alpha_{\tilde a_0}(1/2,(0,1])\subset \overset{\circ}I(\tilde a_1)$.
 \item
 For $z\in f^{-\mu}(p_0)\cap\overline S$, if $z\in I(\tilde\alpha_{\tilde a_0}(n,1))$,
 then $z\in I^{2n-1}(\tilde a_1)$.

 \item
 For $z\in f^{-\mu}(p_1)\cap\overline S$, if $z\in I(\tilde\alpha_{\tilde a_0}(n+1/2,1))$,
 then $z\in I^{2n+1}(\tilde a_0)$.
\end{itemize}
\end{lemma}

\begin{lemma}\label{lem:10-1}
Let $E$ be a connected component of $f^{-\mu}(C_0)$.
Let $\tilde a_0\in f^{-\mu}(a_0)\cap E$ and $\tilde a_1\in f^{-\mu}(a_1)\cap E$  such that $d(\tilde a_0,\tilde a_1)=1$.
Let $z_1\in f^{-\mu}(p_1)\cap I(\tilde a_1)$ and
let $l'$ be the connected component of $f^{-\mu}(l)$ containing $z_1$.
 Then the following hold:
 \begin{enumerate}
   \item If $c\in l'$, then there exist a connected component $E'$ of
   $f^{-\mu}(C_0)$ and $\tilde a_0'\in I(c)\cap E',\tilde a_1'\in I(z_1)\cap E'$
   such that $d(\tilde a_0',\tilde a_1')=1$.
\item If $c\in l'$, then there exists an odd integer $n$ such that $\tilde a_1\in I^n(c)$.
\item
If $c\notin l'$, then $\tilde a_0\in I^n(z_1)$ for any odd integer $n$.
 \end{enumerate}
\end{lemma}

\begin{proof}[Proof of Lemma \ref{lem:10-1}]
  We may consider a connected component $l'$ to be a graph with vertex set
   $V=f^{-\mu}(\{p_0,p_1\})\cap l'$.
   Note that every edge $e$ of $l'$ is an inverse image of $l$, that is,
   $f^\mu:e\to l$ is bijective.

   The first key to the proof is that $c$ is the only critical point of $f^\mu$ in
   $f^{-\mu}(p_0)$.
   The second key  is that if $z\in f^{-\mu}(\{p_0,p_1\})\cap l'$ is not
   a critical point of $f^\mu$,
   then $z$ is an endpoint of $l'$.

1.
For any edge $e$ of the graph $l'$, there uniquely exist an connected component
of $E'$ and $\tilde a_0'\in f^{-\mu}(a_0)\cap E',\tilde a_1'\in f^{-\mu}(a_1)\cap E'$
 such that $d(\tilde a_0',\tilde a_1')=1$ and $I(\tilde a_0')\cap e\ne\emptyset,
 I(\tilde a_1')\cap e\ne\emptyset$.
Indeed,  let $\Gamma$ be the connected component of
 $f^{-\mu}(L_{p_0}\cup l\cup L_{p_1}\cup [a_0,a_1])$ including $e$.
 Then the required $\tilde a_0',\tilde a_1'$ are the points of
 $f^{-\mu}(a_0)\cap\Gamma,f^{-\mu}(a_1)\cap\Gamma$.

 We have to show that the graph $l'$ has an edge with endpoints $z_1$ and $c$.
 If the graph $l'$ has no edge containing both $c$ and $z_1$, then each edge containing $z_1$
 has an endpoint which is not a critical point of $f^\mu$, and $l'$ does
 not contain $c$.
 Thus we have an edge $e$ with endpoints $z_1$ and $c$.

 2.
 Let $\tilde\alpha_{\tilde a_0}:\mathbb R\times(0,1]\to S$ be as in Lemma \ref{lem:10-0}.
 Note that $S$ is a connected component of $f^{-\mu}(\overline{U_l}-l)$.
 We will show that $c\in \overline S$.
To this end, we see  $c$  belongs to
$l_0:=\overline{S}-S$, which is a subgraph of $l'$.
 To the contrary, assume that $c\notin l_0$.
 Every point of $f^{-\mu}(p_0)\cap l_0$ is not a critical point and is an endpoint of $l'$.
 Therefore by the monodromy theorem, the covering $f^\mu:S\to\overline{U_l}-l$ is
 extended to $f^\mu:S'\to \overline{U_l}-\{p_1\}$ as a covering.
Then $S'=S\cup l_0-\{z_1\}$, which is a punctured disc.
 This means $l'=l_0$, and a contradiction to the fact $c\in l'$.
 Thus $c\in l_0$, and $\tilde a_1\in I^n(c)$ for some $n\in \mathbb Z$ by Lemma
  \ref{lem:10-0}.

  3. By a similar argument to the above, we can see that $z_1$ is the only
  point of $f^{-\mu}(p_1)\cap l'$.
    Thus the claim is true by Lemma \ref{lem:10-0}.
\end{proof}

 We define a new system of paths by
 $\mathcal L'=\{T^n(L_p)\}_{p\in P_0}$ for some non-zero integer $n$.
 The integer $n$ will be determined below.
Set $I_i'=I'(c_i)$ for $\mathcal L'$ similarly to $I_i$.
 We show that $I_i'$ joins $E_i$ and $E_{i+1}$ if $i<k-1$ or
 $\mathrm{D}_{\mathcal L}(\mathcal C)>1$.
 Indeed, $I_i'= I_i$  if $f^\mu(c_i)\notin\{p_0,p_1\}$.
We know that  $c=c_k$ is the only point in $\{c_i\}$
 satisfying $f^\mu(c_i)=p_0$.
Suppose $f^\mu(c_i)=p_1$.
 If $c_i$ and $c$ belong to the same connected component of $f^{-\mu}(l)$,
 then $i=k-1$  by the minimality of $\mathrm{Len}(\mathcal C)$ and
 $\mathrm{D}_{\mathcal L}(\mathcal C)=1$ by the minimality of
 $\mathrm{D}(\mathcal C)$ by 1 of Lemma \ref{lem:10-1}.
 Equivalently, if $i< k-1$ or $\mathrm{D}_{\mathcal L}(\mathcal C)>1$,
 then the connected component of $f^{-\mu}(l)$ containing $c_i$ does
 not contain $c$,
 and so we can see that $I_i'$ joins $E_i$ and $E_{i+1}$ for any $n$ by
 3 of Lemma \ref{lem:10-1}.

Let $\tilde a_0\in E_k\cap I_k,\tilde a
=\zeta_{\tilde a_0}(\mathrm{D}_{\mathcal L}(\mathcal C)/m)\in I_{k-1}\cap E_k$
be the points we have taken to define
$\mathrm{D}_{\mathcal L}(\mathcal C)$.
Let $\tilde a_1=\zeta_{\tilde a_0}(1/m)\in f^{-\mu}(a_1)\cap E_k,
\tilde b=\zeta_{\tilde a_0}(D_{\mathcal L}(\mathcal C)/m-1/m)$, and
 $\tilde a'\in E_{k-1}\cap I_{k-1}$.

Case 1.
Suppose $\mathrm{D}_{\mathcal L}(\mathcal C)=1$.
Then $\tilde a'\in f^{-\mu}(a_1), \tilde a=\tilde a_1,\tilde b=\tilde a_0$.
Using 2 of Lemma \ref{lem:10-1}, we have an odd integer $n$ such that
$\tilde a'\in I'(c)$.
We have
  $$\mathcal{C}'=(\{E_0,E_1,\dots,E_{k-1}\},\{c_0,c_1,\dots,c_{k-2},c\})$$
 is a chain for $\mathcal L'=\{T^{n}(L_p)\}_{p\in P_0}$, and so
 $\mathrm{Len}(\mathcal C')=\mathrm{Len}(\mathcal C)-1$.

Case 2.
Suppose  $\mathrm{D}_{\mathcal L}(\mathcal C)>1$.
We set  $n=-1$.
As we have seen above, $\mathcal C$ is a chain for $\mathcal L'$.

-\
If $c_{k-1}\in f^{-\mu}(p_1)$,
then  $\tilde b\in f^{-\mu}(a_0)$.
By the minimality of $\mathrm{D}_{\mathcal L}(\mathcal C)$,
$c_{k-1}$ and $c$ do not belong to the same connected component
of $f^{-\mu}(l)$.
Thus $\tilde b\in I'(c_{k-1})$ and
 $\tilde a_1\in I'(c)$ by  3 of Lemma \ref{lem:10-1} and Lemma \ref{lem:10-0}.
 Thus $\mathrm{D}_{\mathcal L'}(\mathcal C)
 =\mathrm{D}_{\mathcal L}(\mathcal C)-2$.

 -\
If $c_{k-1}\notin f^{-\mu}(p_1)$,
then $\tilde a\in I'(c_{k-1})$ and
 $\tilde a_1\in I'(c)$ by Lemma \ref{lem:10-0}.
 Thus $\mathrm{D}_{\mathcal L'}(\mathcal C)
 =\mathrm{D}_{\mathcal L}(\mathcal C)-1$.

This completes the proof of  Lemma \ref{lem:10}.
\qed

\medskip

\section{The shift locus}
In this section, we prove several results on the parameter spaces.
Hereafter we assume $f$ has the unique solitary fixed point $\infty$
without loss of generality.

\medskip

\noindent { \bf Definition 3.}
 The {\em shift locus} $S_d$ is the set of d-Cantor
 hyperbolic rational maps
 of degree $d$.
 The {\em strong shift locus} $T_d$ is the set of s-Cantor hyperbolic
 rational maps.
 Let $\mathrm{Pol}$ denote the set of polynomial maps.

\begin{theorem}\label{th:pol}
$T_d\cap\mathrm{Pol}=S_d\cap\mathrm{Pol}$.
\end{theorem}

\begin{proof}
Let $f\in S_d\cap\mathrm{Pol}$.
Then all critical points belong to the attracting basin of
$\infty$.
Let $U$ be a simple domain of $\infty$ such that
$U$ contains no critical value other than $\infty$.
For every $p\in (f^{-1}(U)-U)\cap P$, we take a simple
arc $\gamma_p\subset f^{-1}(U)-U$ joining $p$
and a point of $\partial U$ such that $\gamma_p\cap\gamma_q
=\emptyset$ if $p\ne q$.
Note that $f^{-n}(U)$ is connected for every $n>0$ since
$f^{-1}(\infty)=\{\infty\}$.
We inductively take a simple arc $\gamma_p\subset
f^{-n}(U)-U$ for
$p\in (f^{-n}(U)-f^{-n+1}(U))\cap P$ joining $p$ and
a point of $\partial U$ such that $\gamma_{f(p)}
=f(\gamma_p\cap(f^{-n}(U)-f^{-1}(U)))$ and
$\gamma_p\cap\gamma_q
=\emptyset$ if $p\ne q$.
Thus we obtain a simply connected open set
$V=\hat{\mathbb C}-(\overline U\cup
\bigcup_{p\in P-U}\gamma_p)$ which
satisfies $f^{-1}(V)\subset V$.
Then it is easy to see that $f$ is s-Cantor.
\end{proof}

\begin{theorem}\label{th4}
 $T_d\subsetneq S_d$ for $d\ge4$, and $T_2=S_2$.
\end{theorem}

The $d=2$ case is a consequence of the following.
\begin{lemma}
  Let $f\in S_d$ which has exactly two critical points.
  Then $f\in T_d$.
\end{lemma}
\begin{proof}
  Note that both of the critical points have local degree $d$.
 Take $U$ a simple domain of $\infty$.
We take open sets $U_k$ as in Definition \ref{defidomain}.
Suppose that all critical points belong to the immediate
basin of attraction for $\infty$.
If $U_k$ contains at most one critical value,
then $U_{k+1}$ is simply connected.
Let $n$ be the smallest integer such that $U_n$
contains two critical values.
Then $f$  satisfies the condition of s-Cantor
rational map with $\overline D=\hat{\mathbb C}-U_n$.
\end{proof}

To prove the other cases, it is sufficient to show that
there exists a rational map $f\in S_d-T_d$.
We will do this in section \ref{sec:gct}.

\begin{cor}\label{th5}
$T_d$ is  open and dense in $S_d$.
\end{cor}

\begin{theorem}\label{th6}
 $S_d$ is connected.
\end{theorem}

\begin{theorem}\label{th7}
The set of rational maps $f\in T_d$ satisfying the
assumption of Theorem \ref{th:8} with an attracting (not superattracting)
 fixed point $\infty$ such that all the critical points have the local degree two
is connected, and all the maps in the set are qc-conjugate to each other.
\end{theorem}


To prove these three statements, crucial is Theorem \ref{th:8}.
Additionally, we use qc-surgery (quasiconformal surgery).
For the basics of qc-surgery, see \cite{BrFa14}.

{\em Proof of Corollary \ref{th5}.}
It is easy to see that  $S_d$ and $T_d$ are open in the space of
rational maps of degree $d$.
If $f\in S_d-T_d$,
then we add an arbitrary small perturbation to $f$
by qc-surgery
 that makes the map satisfy the assumption of Theorem \ref{th:8}.
\qed

\medskip

{\em Proof of Theorem \ref{th6}.}
We can deduce this theorem from Theorem \ref{th7}.
However we give another proof here.

Let $f\in S_d$.
We construct a path in $S_d$ which connects $f$ and
some map in $S_d\cap\mathrm{Pol}$.
First by  a perturbation, we obtain $f_0$ satisfying the condition
 of Theorem \ref{th:8}, and moreover we can assume that
 any critical point other than $\infty$ has local degree two.

 For simplicity, we omit the subscript of $f_0$.
 By the proof of Theorem \ref{th:8}, we have
 a topological disc $W$ containing $\infty$ such that
 $\overline W\subset f^{-1}(W)$, and such that $A=W'-W$ contains all the
 critical values other than $\infty$, where $W'$ is the connected
 component of $f^{-1}(W)$ including $W$.
If $f$ is not a polynomial, then $f^{-1}(W)$ is not connected.

 Claim.
 There exists a topological disc $\tilde W$ with smooth boundary such that
 $W\subset \tilde W \subset W'$ and such that $f^{-1}(\tilde W)$
 is a topological disc.

 {\em Proof of Claim.}
 Let $P_0\subset A$ be the set of critical values other than $\infty$.
 Take distinct points $a_p\in \partial W$ for $p\in P_0$, and
 disjoint simple paths  $L_p \subset A\cup\partial W$
 joining $p\in P_0$ and $a_p$.
 Then $f^{-1}(\bigcup_{p\in P_0} L_p\cup \partial W)$
 is a connected set including $\partial W'$.
 Take a minimal subset $P_0'\subset P_0$ such that
 $f^{-1}(\bigcup_{p\in P_0'}L_p\cup \partial W)$ is connected.
 Then $f^{-1}(\bigcup_{p\in P_0'}L_p\cup W)$ is connected.
 We take a topological disc $\tilde W$ with smooth boundary  such that
 $W'\supset\tilde W\supset \bigcup_{p\in P_0'}L_p\cup W$
 and $\tilde W\cap P_0=P_0'$.
 This completes the proof of Claim.

 From Lemma \ref{lem9}, we have a path $f_t, 0\le t\le1$ between $f_0=f$
 and a rational map $f_1$ such that $f_t(\infty)=\infty$,
 such that there is a continuous family
 of topological discs between $U_0=f^{-1}(\tilde W)$ and $U_1$
 with $f_t(U_t)\subset U_t$,
 such that all of $f_t:\hat{\mathbb C}-f_t^{-1}(U_t)\to \hat{\mathbb C}-U_t$ are
 topologically conjugate to each other, and
 such that $U_1$  contains no critical point other than $\infty$.
It is easy to see that $f_t\in S_d$ and $f_1$ is a polynomial.
The proof is completed, since the shift locus of $d$-dimensional
polynomial is connected (\cite{DeMarcoPilgrim11}).
\qed

\begin{lemma}\label{lem9}
 Let $f_0$ be a rational map with a (super)attracting fixed point $\infty$.
 Suppose that there exists a topological disc $U_0$ containing $\infty$
 with smooth boundary $\partial U_0$ such that
 $f_0|_{U_0}:U_0\to f(U_0)\subset U_0$ is a proper map of degree
 $m\ge2$ and $\partial U_0$
 contains no critical point.
 Then  there exists a continuous family $f_t,0\le t\le1$ of rational maps
 with a (super)attracting fixed point $\infty$  and
 a continuous family of topological discs $U_t$ containing $\infty$
 such that $f_t|_{U_t}:U_t\to f_t(U_t)\subset U_t$,
 such that all of $f_t:\hat{\mathbb C}-f_t^{-1}(U_t)\to \hat{\mathbb C}-U_t$ are
 topologically conjugate to each other, and
 such that $U_1$  contains no critical point of $f_1$ other than $\infty$.
\end{lemma}

\begin{proof}
 Take a topological disc $f^2(U_0)\subset V\subset f(U_0)$
 with smooth boundary such that
 $A:=f(U_0)-\overline V$ is an annulus containing no critical value.
Set $V':=U_0\cap f^{-1}(V)$ and $A':=U_0\cap f^{-1}(A)$.
 Then $f:A'\to A$ is a degree $m$ covering.

 Let $\phi:V\to\mathbb D$ and $\phi':V'\to\mathbb D$ be Riemann
 maps with $\phi(\infty)=\phi'(\infty)=0$.
 Then $h=\phi\circ f\circ\phi'^{-1}:\mathbb D\to\mathbb D$ is a proper
 holomorphic map with $h(0)=0$.
 Thus $h$ is a Blaschke product
 $e^{i\theta}\prod_{k=1}^m\frac{z-a_k}{1-\overline{a_k}z}$,
 where $\{a_k\}=\phi'\circ f^{-1}(\infty)$.

 Set
 $$h_t(z)=e^{i\theta}\prod_{k=1}^m\frac{z-(1-t)a_k}{1-(1-t)\overline{a_k}z},$$
 for $0\le t\le1$.
 Define a continuous family of smooth branched coverings
 $g_t:\hat{\mathbb C}\to\hat{\mathbb C}$
 satisfying $g_t(z)=\phi^{-1}\circ h_t\circ \phi'$ for $z\in V'$,
 $g_t(z)=f_0(z)$ for $z\in \hat{\mathbb C}-U_0$, and
 that  $g_t:A'\to A$ is a degree $m$ covering map.

 Since for any point $z\in \hat{\mathbb C}$ its orbit passes through $A'$
 at most once, by Shishikura principle we obtain the required $f_t$.
\end{proof}

\begin{proof}[Proof of Theorem \ref{th7}]
  Let $f_1,f_2$ be rational maps satisfying the assumption of the theorem.
We construct a qc-conjugacy  $h$ between $f_1$ and $f_2$ in the sequel.
From this, we can construct a path of qc-equivalent rational maps which
 connects $f_1$ and $f_2$ (cf. Theorem 2.9 of \cite{McSu98}).

By Theorem \ref{th:8},
we can assume  that for $i=1,2$, there is a
topological disc $W_i$ containing $\infty$ with smooth boundary
 such that $A_i:=W_i'-\overline W_i$ includes all the critical values,
 where $W_i'$ is the
 connected component of $f_i^{-1}(W_i)$ containing $\infty$.

 To show the existence of $h$, it is sufficient to construct a diffeomorphism
 $m:\overline {A_1}\to\overline{A_2}$ such that
 there exists a lift $\tilde m:\overline{A_1'}\to\overline{A_2'}$
 with $f_2\circ \tilde m=m\circ f_1$ on $\overline{A_1'}$, where
 each $A_i',i=1,2$ is the connected component of $f_i^{-1}(A_i)$
 adjacent to $A_i$.
 Recall that $A_i'$ contains all the critical points.
If we have such a diffeomorphism $m$, then we define a
qc map $h$ as follows:
First define $h=f_2^k\circ m\circ f_1^{-k}$ on $f_1^k(A_1),k=1,2,\dots$.
Then we have $h:\overline{W_1}\to\overline{W_2}$.
Secondly extend $h$ on $f^{-k}(\overline{W_1}),k=1,2,\dots$
 to satisfy $f_2\circ h=h\circ f_1$.
This extension is guaranteed by the existence of $\tilde m$, and
  we have $h$ on the Fatou set.
Finally it is evident that $h$ can be continuously extended to
the Julia set by construction.

The proof is completed by Lemma \ref{lem11-2}.
\end{proof}

\begin{lemma}\label{lem11-1}
  Let $A\subset\mathbb C$ be a topological annulus with smooth boundary,
  $A'\subset\mathbb C$
  a $2d$-connected domain with smooth boundary,
   and $f:A'\to A$ a branched covering of degree $d$
   with $2d-2$ critical point $p_1,p_2,\dots p_{2d-2}\in A'$ of degree two.
We may consider the subscripts $i=1,2,\dots,2d-2$ to be elements of $\mathbb Z/(2d-2)$.
   Suppose $q_i=f(p_i),i=1,2,\dots,2d-2$ are distinct.
   Take $x\in A',y=f(x)\in A$.
   Then renumbering $p_i$'s if necessary, there exist simple paths $l_i\in A,i=1,2,\dots 2d-2$
   satisfying:
   \begin{enumerate}
     \item $l_i$ joins $y$ and $q_i$ ($i=1,2,\dots 2d-2$), and
      they are disjoint except at $y$,
    \item      $l_i,i=1,2,\dots 2d-2$
    are arranged counter-clockwise around $y$,
    \item    $p_i\in l_i'$, where $l_i'$ is the connected component  of
       $f^{-1}(l_i)$ containing $x$, and
\item each of $l_{2k-1}'\cup l_{2k}',k=1,2,\dots,d$ is a closed curve.
   \end{enumerate}
\end{lemma}

\begin{proof}[Proof of Lemma \ref{lem11-1}]
First we take a system of paths $\mathcal L=\{l_i:i=1,2,\dots,2d-2\}$ satisfying
1.
Set $L=\bigcup_{i=1}^{2d-2} l_i$.
Modifying it step by step, we obtain required paths.

Observe that $f$ is one-to-one  on each boundary component of $A'$.
We can show that if $\mathcal L=\{l_i\}$ satisfies 1, then $f^{-1}(L)$ is connected
and $f^{-1}(A-L)$ has exactly $d$ connected components.
Indeed, if $f$ is of degree $q$ on some connected component
$D\subset f^{-1}(A-L)$ and $p$ is the number of connected components of
$\partial D\cap f^{-1}(L)$,
then the Euler characteristic fulfills
$2-p-2q=\chi(D)=q\chi(A-L)=-q$, and so $q=1$ and $p=1$.

Moreover, observe that if there are distinct connected components $D,D'\in f^{-1}(A-L)$
 such that $\overline D\cap\overline{D'}$ includes a nontrivial curve
 (i.e. neither empty nor finite points), then
$\overline D\cap\overline{D'}$ includes a connected component of
$f^{-1}(l_i)$ for some $i$ which contains $p_i$.

In the following, we denote by $D,D'$  connected components of $f^{-1}(A-L)$.

Step 1.
Let $\mathcal L=\{l_i\}$ be a system of paths satisfying 1.
We modify $\mathcal L$ to satisfy 1 and 3.
Suppose that there exist  a $D$
with  $x\in \partial D$ and a critical point $p_i\in \partial D$ with
$p_i\notin l_i'$.
Then we take a simple path $\tilde l_i'\subset \overline D$ between $x$ and $p_i$,
 and replace $l_i$ with $f(\tilde l_i')$.
By  repeating this procedure, we obtain $\mathcal L$ such
that $p_i\in l_i'$ for any $p_i,D$ with $x,p_i\in\partial D$.
This means that for any $D$ with $x\in \partial D$ and for any $i$,
either $p_i\in\partial D\cap f^{-1}(l_i)=l_i'$ or $p_i\notin\partial D\cap f^{-1}(l_i)$.
Therefore any  $D$ satisfies $x\in\partial D$.
Indeed, if $x\in\partial D$ and $\partial D\cap\partial D'$ includes a
nontrivial  curve,
then there exists $p_i$ such that $p_i\in l_i'\subset\partial D\cap\partial D'$,
and so $x\in\partial D'$.
Hence  we have $\mathcal L$ with 1 and 3.

Step 2.
By renumbering, we have $\mathcal L$ satisfying 1, 2, and 3.
We modify $\mathcal L$ to satisfy 1, 2, 3, and 4.
We can show that for each $i$, there exists $j\ne i$ such that
$l_i'\cup l_j'$ is a closed curve.
Indeed, it is sufficient to show that for each $x'\in f^{-1}(y)-\{x\}$,
there exist two integers $i,j\in\{1,2,\dots,2d-2\}$ such that $l_i'$
and $l_j'$ joins $x$ and $x'$.
Since $f^{-1}(L)$ is connected, every $x'\in f^{-1}(y)-\{x\}$
has at least one $i$ such that $l_i'$ joins $x$ and $x'$.
If $l_i'$ is the only one which joins $x$ and $x'$, then there exists
a $D$ such that a neighborhood
of $p_i$ is included in the interior of $\overline{D}$, which contradicts
the injectivity of $f:D\to A-L$.

Let $i,j$ be such that $i-j\ne0,\pm1$ and such that
$l_i'\cup l_j'$ is a closed curve.
Then there exists a  $D$ such that
$l_i',l_j'\subset\partial D$.
We can take a simple path $\tilde l_j'$ in $\overline D$ joining $x$ and $p_j$
such that $f(\tilde l_j')$ is either between $l_i$ and $l_{i+1}$ or
between $l_{i-1}$ and $l_i$ in the cyclic order around $y$.
We replace $l_j$ with $f(\tilde l_j')$.
Then the new $l_i'\cup l_j'$ is also a closed curve by the fact
proved in the previous paragraph.
Renumber $\mathcal L$.
By  repeating this procedure, we obtain a required $\mathcal L$.

This completes the proof of Lemma \ref{lem11-1}.
\end{proof}

\begin{lemma}\label{lem11-2}
  Let $f_1:A_1'\to A_1$ and $f_2:A_2'\to A_2$ be branched coverings
satisfying the assumption of Lemma \ref{lem11-1}.
Then they are  isomorphic to each other, that is, there exist diffeomorphsims
$m:A_1\to A_2$ and $\tilde m:A_1'\to A_2'$ such that $f_2\circ \tilde m=m\circ f_1$.
\end{lemma}

\begin{proof}[Proof of Lemma \ref{lem11-2}.]
  By Lemma \ref{lem11-1}, we have a system of paths $\mathcal L_1,\mathcal L_2$
  for $f_1,f_2$ respectively.
Take a diffeomorphism $m:A_1\to A_2$ which carries $\mathcal L_1$ to $\mathcal L_2$.
Then the image of $(m\circ f_1)_*:\pi_1(A_1'-f^{-1}(Q_1),x_1)\to\pi_1(A-Q_2,y_2)$
coincides with the image of $(f_2)_*:\pi_1(A_2'-f^{-1}(Q_2),x_2)\to\pi_1(A-Q_2,y_2)$,
where $Q_i,i=1,2$ are the sets of critical values for $f_i$, and
$x_i,y_i,i=1,2$ are the basepoints in $A_i'$ and $A_i$.
Indeed, the images are determined by
only the graph structures of $(m\circ f_1)^{-1}(L_2)
=f_1^{-1}(L_1)$ and $f_2^{-1}(L_2)$.

Now the existence of $\tilde m$ is assured by the theory of covering spaces.
\end{proof}

\begin{remark}\hspace{1mm}
\label{rem4}

  \begin{enumerate}
    \item The limitation ``renumbering $p_i$'s'' in Lemma \ref{lem11-1} can be deleted.
    In fact, any permutation $\tau:\{q_i\}\to\{q_i\}$ can be extended to
    a homeomorphism $m:A\to A$ with a homeomorphism
    $\tilde m:A'\to A'$ such that $m\circ f=f\circ\tilde m$.
    The isotopy class of $m$ relative to $\{q_i\}$ is not unique.
    \item
    Take a generating set $\{g_i:i=1,2,\dots,2d-2\}\cup\{g'\}$
    of $\pi_1(A-\{q_i\},y)$
    in Lemma \ref{lem11-1} as follows: $g_i=[\gamma_i],g'=[\gamma']$,
    where $\gamma_i$ is a loop
    with basepoint $y$ which is the boundary of a topological
    disc $U_i\subset A$ such that $\overline U_i\cap L=l_i$, and
    $\gamma'$ is a nontrivial loop of the annulus $A$.
Then the image of $f_*:\pi_1(A'-f^{-1}(\{q_i\}),x)\to\pi_1(A-\{q_i\},y)$
      is generated by
     $\{g_i^2,g_{2k-1}g_{2k},g_{2k-1}g_lg_{2k},g',g_{i}g'g_{i}\}$, where
     $i,k,l$ run over integers such that $1\le i\le 2d-2,1\le k\le d+1,$ and $l\ne 2k-1,2k.$
  \end{enumerate}
\end{remark}

By a similar argument to  Theorem \ref{th7}, we obtain the following:

\begin{theorem}
  \label{th7.1}
  Let $f,g$ be rational maps satisfying the assumption of Theorem \ref{th:8}
with a (super)attracting fixed point $\infty$  such that all the critical points
other than $\infty$ have the local degree two.
Suppose that there exists a qc-map
$\phi:B_f\to B_g$ with $\phi\circ f=g\circ \phi$, where
$B_f,B_g$ are neighborhoods of $\mathrm{Cr}_f\cup {P_f}\cup\{\infty\}, \mathrm{Cr}_g\cup {P_g}\cup\{\infty\}$ respectively.
Then we have:
\begin{enumerate}
  \item (Theorem \ref{th7} and Remark \ref{rem4}(1))
  If $\infty$ is attracting, then $\phi$ can be extended to
  a qc-conjugacy on the whole sphere.
  \item If $\infty$ is supperattracting, then $\phi$ can be extended to
  a qc-conjugacy on any bounded set.
\end{enumerate}
\end{theorem}

We leave the following problem to the reader.

\medskip

\noindent
{\bfseries  Problem }
Relax the restriction of local degree in Theorem \ref{th7.1}.

\begin{remark}
Theorem \ref{th7.1} fails if the assumption of Theorem \ref{th:8} is
omitted.
Indeed, we give an example.
Let $f\in S_2$ be a rational map with attracting fixed point $\infty$
and two critical points $c_1,c_2$ such that $f^2(c_1)=f^2(c_2)$.
Then there  exists no homeomorphsim $\phi:\hat{\mathbb C}\to\hat{\mathbb C}$
such that $\phi\circ f=f\circ\phi$ and $\phi(c_1)=c_2$.
The detail is left to the reader.
\end{remark}

 \section{An example of a rational map which is Cantor but not strongly Cantor}\label{sec:gct}

In this section, we complete the proof of Theorem \ref{th4}.

\subsection{Basic facts}\label{sec6-1}

 \begin{defi}\label{defsemidirectp}
  Let $G,T$ be  groups with a homomorphism
  $\varphi:T\to\mathrm{Aut}G$.
  Then we have the semidirect product $G\rtimes_\varphi T$, namely,
  $G\rtimes_\varphi T=\{(g,\tau)\,|\,g\in G,\tau\in T\}$,
and for $(g,\tau),(g',\tau')\in G\rtimes_\varphi T$, the product
$(g'',\tau'')=(g,\tau)(g',\tau')$ is defined by
$g''=g\varphi(\tau)(g')$, $\tau''=\tau\tau'$.
We have an inclusion $i_1:G\to G\rtimes_\varphi T,
i_1(g)=(g,1)$ and a projection $p_2:G\rtimes_\varphi T
\to T,p_2(g,\tau)=\tau$, and an
exact sequence $1\to G\to G\rtimes_\varphi T\to T\to1$.

  For a set $X$, let $\mathfrak{S}(X)^{\mathrm{op}}$ be the opposite
  group of the symmetric group $\mathfrak{S}(X)$.
  Namely, $\mathfrak{S}(X)^{\mathrm{op}}$ consists of all bijective
  self-maps on $X$, and for $\tau,\tau'\in \mathfrak{S}(X)^{\mathrm{op}}$,
  we set $\tau\tau'=\tau'\circ\tau$.
  For a group $G$ and a set $X$, we have a homomorphism
  $\varphi=\varphi_{G,X}:\mathfrak{S}(X)^{\mathrm{op}}\to\mathrm{Aut} (G^X)$
  defined
  by $\varphi(\tau)(a)=a\circ \tau$.
We write $M(G,X)=G^X\rtimes_\varphi \mathfrak{S}(X)^{\mathrm{op}}$.

  For convenience of calculation, an element $(a,\tau)\in
M(G,X)$ is
expressed in the form of formal series
$$\sum_{x\in X} a(x)\cdot (x,\tau(x)).$$
Define the product of two terms of such a series by
$$[g\cdot(x,y)][g'\cdot(x',y')]=\left\{
\begin{array}{cc}
gg'\cdot(x,y')&\text{ if }y=x',\\
0&\text{ if }y\ne x',
\end{array}\right.$$
where $g,g'\in G,x,y,x',y'\in X$.
Then we recover the operation of $M(G,X)$ using distributive law.

For subgroups $Q\subset G^X$ and $T\subset\mathfrak{S}(X)$
such that $Q$ is $T$-invariant, the subgroup
$Q\rtimes T=\langle i_1(Q),p_2^{-1}(T)\rangle
\subset M(G,X)$ is well-defined.

\end{defi}

\begin{defi}\label{defi8}
 Let $G_0=\pi_1(\hat{\mathbb C}-P,\bar x)$ be the fundamental group, and
 $\hat G_0=\hat\pi_1(\hat{\mathbb C}-P,\bar x)=\pi_1(\hat{\mathbb C}-P,\bar x)
 /\ker\alpha_\infty$ the iterated monodromy group (see Definition
 \ref{defmono}).
   The monodromy  $\alpha=\alpha_f$ descend to the right action of
 $\hat G_0$ on $f^{-1}(\bar x)$, which
 is denoted by the same notation $\alpha$.
 Therefore
 $\alpha:\hat G_0\to\mathfrak{S}(f^{-1}(\bar x))^{\mathrm{op}}$
 is a homomorphism.
 Let $G_k=M(G_0,f^{-k}(\bar x))$ and $\hat G_k=M(\hat G_0,f^{-k}(\bar x))$,
 $k=1,2,\dots$.
This group and the homomorphism $f_r^*$ below are
 used by the author in \cite{Ka03},\cite{Ka01}.

 Let $r=(l_i)_i$ be a radial, and  $\chi_r:\mathtt{Word}(k)\to f^{-k}(\bar x)$
 the bijection determined by
 $w\mapsto l_w(1)$ using the notation  in Definition \ref{deftree}.
 We identify $f^{-k}(\bar x)$ with $\mathtt{Word}(k)$ by $\chi_r$, and
 also  $G_k$ with $M(G_0,\mathtt{Word}(k))$.
We consider $\alpha$ to be a right action of $G_0$ on
$\mathtt{Word}(1)=\{1,2,\dots,d\}$.
This action depends on $r$.

For a loop $\gamma:([0,1],\{0,1\})\to (\hat{\mathbb C}-P,\bar x)$,
let denote by $[\gamma]\in G_0$ the equivalence class including $\gamma$.
 We obtain (non-homomorphic) mappings $\beta_{r,i}:G_0\to G_0$,
  $i=1,2,\dots,d$
 defined by
 $\beta_{r,i}([\gamma])=[l_iL_i(\gamma)l_{\alpha([\gamma])(i)}^{-1}]$,
 where $L_i(\gamma)$ is the lift of $\gamma$
 defined in Definition \ref{deftree}.
 Note that $$\beta_{r,i}[\gamma\gamma']
 =[l_iL_i(\gamma\gamma')l_{j'}^{-1}]
 =[l_iL_i(\gamma)l_{j}^{-1}
 l_{j}L_j(\gamma')l_{j'}^{-1}]
 =\beta_{r,i}([\gamma])\beta_{r,j}([\gamma']),$$
 where $j=\alpha([\gamma])(i)$ and $j'=\alpha([\gamma\gamma'])(i)
 =\alpha([\gamma'])(j)$.
From this, we have the homomorphism $f_r^*:G_{k}\to G_{k+1}$  determined
by
$$f_r^*:\sum_{w\in \mathtt{Word}(k)} a(w)\cdot (w,\tau(w))\mapsto
\sum_{i=1}^d\sum_{w\in \mathtt{Word}(k)}
\beta_{r,i}(a(w))\cdot(iw,\alpha(a(w))(i)\tau(w)).$$
Particularly, for $k=0$,
$$f_r^*:g\mapsto
\sum_{i=1}^d
\beta_{r,i}(g)\cdot(i,\alpha(g)(i)).$$
   We write  $\beta_{r,w}=\beta_{r,i_1}\circ\beta_{r,i_2}
   \circ\cdots\circ\beta_{r,i_k}$ for $w=i_1i_2\cdots i_k\in \mathtt{Word}(k)$.
It is easy to see that
$\beta_{r,w}([\gamma])=[l_wL_w(\gamma)l_{\alpha_{f^k}([\gamma])(w)}^{-1}]$.
\end{defi}

\begin{prop}
 Let $r$ be a radial.
 The action $\alpha_{f^k}$ of $G_0$ on $\mathtt{Word}(k)\cong f^{-k}(\bar x)$
 is inductively expressed as
 $$\alpha_{f^k}(g)(i_1i_2\cdots i_k)=\alpha_f(\beta_{r,i_2\cdots i_k}(g))(i_1)
 \alpha_{f^{k-1}}(g)(i_2\cdots i_k)
 ,$$
 and
 $$\overbrace{f_r^*\circ f_r^*\circ\cdots\circ f_r^*}^{k\text{ times}}
 (g)
     =\sum\limits_{w\in \mathtt{Word}(k)}
     \beta_{r,w}(g)
     \cdot(w,\alpha_{f^k}(g)(w)).$$
\end{prop}

\begin{proof}
By the definition of $f_r^*$, it is sufficient to show
$$\alpha_{f^{k+1}}([\gamma])(iw)=\alpha(\beta_{r,w}([\gamma]))(i)
\alpha_{f^k}([\gamma])(w)$$ for $i\in \mathtt{Word}(1)$ and
 $w\in \mathtt{Word}(k)$.
Suppose $\alpha_{f^{k+1}}([\gamma])(iw)=ju,j\in \mathtt{Word}(1),
u\in \mathtt{Word}(k)$.
Then $L_{iw}(\gamma)(0)=l_{iw}(1)$ and  $L_{iw}(\gamma)(1)=l_{ju}(1)$.
Composing $f$, we have
$L_{w}(\gamma)(0)=l_{w}(1)$ and  $L_{w}(\gamma)(1)=l_{u}(1)$.
Thus $\alpha_{f^k}([\gamma])(w)=u$.
On the other hand, $L_{iw}(\gamma)(1)=l_{ju}(1)$ implies that
the endpoint of $L_i(l_wL_w(\gamma)l_u^{-1})$ is $l_j(1)$.
Thus $\alpha(\beta_{r,w}([\gamma]))(i)=j$.
\end{proof}

It is easily seen that $\beta_{r,i}:\hat G_0\to\hat G_0$ is well-defined and
the homomorphism $f_r^*:G_k\to G_{k+1}$ descends to
$f_r^*:\hat G_k\to \hat G_{k+1}$ (we use the same symbol) as the diagram
$$
  \raisebox{-0.5\height}{\includegraphics[trim=0 745 360 40,clip]{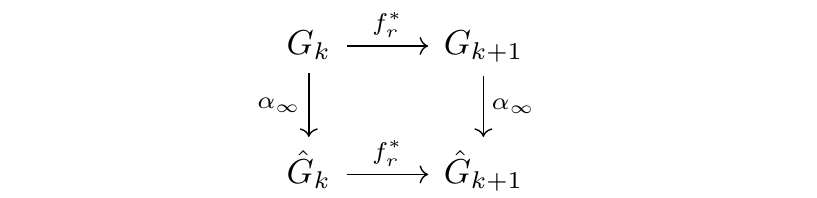}}
$$
commutes.
 Let $r'=(l_i')$ be a radial other than $r$ with the same numbering of
 $f^{-1}(\bar x)$, i.e. $l_i(1)=l_i'(1)$.
Let $h_i=[l_i'l_i^{-1}]\in G_0$.
Then $\beta_{r',i}([\gamma])=h_i\beta_{r,i}([\gamma])
h_{\alpha([\gamma])(i)}^{-1}$.
Setting $h=((h_1,h_2,\dots,h_d),\mathrm{id})=\sum_{i=1}^d h_i\cdot(i,i)\in G_1$,
 we can describe $f_{r'}^*:G_0\to G_1$ as
 $$f_{r'}^*:g\mapsto hf_{r}^*(g)h^{-1}.$$
 However, since $\chi_r\ne \chi_{r'}$ on $\mathtt{Word}(k),k\ge2$,
  the relation between $f_r^*$ and $f_{r'}^*$ on $G_k,k\ge 1$ is
  complicated.

\begin{remark}
 Let $\gamma:([0,1],\{0,1\})\to (\hat{\mathbb C}-P,\bar x)$ be  a loop.
 There is a one-to-one correspondence between
 the lifted loops of $\gamma$ and the periodic cycles of $\alpha([\gamma])$.
 Namely, $(i_1,i_2,\dots,i_m)$ is a periodic cycle of $\alpha([\gamma])$,
 if and only if
 $\tilde \gamma=L_{i_1}(\gamma)L_{i_2}(\gamma)\cdots L_{i_m}(\gamma)$
 is a loop and $f\circ \tilde\gamma=\gamma^m$.
\end{remark}

Let us denote by $e$ the identity element of  $\hat G_n$.
\begin{prop}
The homomorphism $f_r^*:\hat G_{k}\to \hat G_{k+1}$ is injective.
\end{prop}

\begin{proof}
Suppose  $\sum_{i=1}^d\sum_{w\in \mathtt{Word}(k)}
 \beta_{r,i}(a(w))\cdot(iw,\alpha(a(w))(i)\tau(w))=e$.
 Then $\tau=\mathrm{id}$, and $\beta_{r,i}(a(w))=e$, $\alpha(a(w))=\mathrm{id}$
 for any $w\in \mathtt{Word}(k)$ and any $i\in \mathtt{Word}(1)$.
 Thus we have $\beta_{r,u}(a(w))=e$ for any
 $u\in \mathtt{Word}$ and $w\in \mathtt{Word}(k)$, and so $\alpha_\infty(a(w))=1$.
\end{proof}

\begin{lemma}\label{lemma4}
Let $f$ be a Cantor hyperbolic rational map of
degree $d$,
and $r=(l_i)_i$ a radial for $f$.
Then the coding map $\phi_r:\Sigma_d\to J$ is one-to-one
if and only if for any $g\in \hat G_0$, there exists $k>0$
such that $(f_r^*)^k(g)\in \{e\}\rtimes \mathfrak{S}(\mathtt{Word}(k))$.
\end{lemma}

\begin{proof}
 Let $\Omega$ be an expanding domain including $r$, and
  set $\hat G_\Omega=\pi_1(\Omega-P,\bar x)/\ker \alpha_\infty\subset \hat G_0$.
Recall that $\hat G_\Omega$ is a finite group (Theorem \ref{th1}).
  We have
   $\beta_{r,w}(\hat G_\Omega)\subset \hat G_\Omega$ for any $w\in \mathtt{Word}$;
   for any $g\in \hat G_0$, there exists $w\in \mathtt{Word}$
    such that $\beta_{r,w}(g)\in \hat G_\Omega$.

 Let $Q=\{(\omega,\omega')\in \Sigma_d\times\Sigma_d:\phi_r(\omega)=\phi_r(\omega')\}$.
 For $q=(\omega,\omega')\in Q$, let us  denote
 $\sigma(q)=(\sigma(\omega),\sigma(\omega'))$,
 $\phi(q)=\phi_r(\omega)=\phi_r(\omega')$,
 and $[\gamma_q]=[l_\omega l_{\omega'}^{-1}]\in\hat G_\Omega$,
 where $\sigma:\Sigma_d\to\Sigma_d$
 is the shift map.
 Note that $l_w\subset\Omega$ for any $w\in\mathtt{Word}$.
 Thus $[\gamma_q]\in \hat G_\Omega$ for any $q\in Q$.

We show $(i\omega,j\omega)\notin Q$ if $i\ne j$.
 Indeed, since the Julia set $J$ contains no critical point, $f^{-1}(z)$ consists of
 exactly $d$ points for $z\in J$.
Hence for $z=\phi_r(\omega)$, the two sets of lifts of $l_\omega$ coincide:
$$\{L_i(l_\omega):i\in\mathtt{Word}(1)\}
=\{L'_y(l_\omega):y\in f^{-1}(z)\},$$
where by $L_y'(l)$ we denotes the lift of $l$ by $f$ such that
 $y$ is one of the endpoints.
This means that $\phi_r(i\omega),i\in\mathtt{Word}(1)$ are all distinct.

For $q=(\omega,\omega')\in Q$, we have
 $\sigma(q)\in Q$ and $(i\omega,j\omega')\in Q$ for every
 $i\in\mathtt{Word}(1)$ and $j=\alpha([\gamma_q])(i)$.
Clearly, $\beta_{r,i}([\gamma_q])=[\gamma_{(i\omega,j\omega')}].$
From this fact, if $q=(\omega,\omega')\in Q$ and $[\gamma_q]=e$, then
$wq=(w\omega,w\omega')\in Q$ and $[\gamma_{wq}]=e$ for any $w\in\mathtt{Word}$.
Therefore if $[\gamma_q]\ne e$, then $[\gamma_{\sigma^k(q)}]\ne e$ for any $k>0$.

We show that the following are equivalent:
\begin{enumerate}
  \item\label{9-1} $\phi_r:\Sigma_d\to J$ is one-to-one.
  \item\label{9-2} $\{[\gamma_q]\in\hat G_\Omega:q\in Q\}=\{e\}$.
  \item\label{9-3}
  $\bigcap_{k>0}\bigcup_{w\in\mathtt{Word}(k)}\beta_{r,w}(\hat G_\Omega)=\{e\}$.
  \item\label{9-4}   For any $g\in \hat G_0$, there exists $k>0$
  such that $(f_r^*)^k(g)\in  \{e\}\rtimes \mathfrak{S}(\mathtt{Word}(k))$.
  \item\label{9-5} There does not exist a non-trivial $g\in \hat G_0$ such that
  $\beta_{r,w}(g)=g$ for some $w\in \mathtt{Word}$.
\end{enumerate}

\ref{9-4} $\iff$ \ref{9-3} $\Rightarrow$ \ref{9-5}  is easy.

\ref{9-5} $\Rightarrow$ \ref{9-3} is obtained by the finiteness of
$G_\Omega$.

\ref{9-2} $\Rightarrow$ \ref{9-1}.
Suppose $\{[\gamma_q]\in\hat G_\Omega:q\in Q\}=\{e\}$.
First we show that if $q=(i\omega,j\omega')\in Q$, then $i=j$.
Indeed, by the fact mentioned above, $\sigma(q)=(\omega,\omega')\in Q$
and $(i\omega,i\omega')\in Q$ since $i=\alpha([\gamma_{\sigma(q)}])(i)$.
Hence $(j\omega',i\omega')\in Q$, and so  $i=j$.

Applying this observation to $\sigma^k(q),k=1,2,\dots$,
we have that $Q$ is the diagonal set, and
 $\phi_r$ is one-to-one.

 \ref{9-5} $\Rightarrow$ \ref{9-2}.
 Suppose $[\gamma_q]\ne e$ for some $q\in Q$.
 Then $[\gamma_{\sigma^k(q)}]\ne e$ for $k\ge0$.
 There exist  $k,k'\in\mathbb N$ such that $k>k'$ and
 $[\gamma_{\sigma^k(q)}]=[\gamma_{\sigma^{k'}(q)}]$,
 since $\hat G_\Omega$ is a finite group.
Hence $g=[\gamma_{\sigma^{k'}(q)}]\ne e$ satisfies
$\beta_{r,w}(g)=g$ for some $w\in\mathtt{Word}(k-k')$.

\ref{9-1} $\Rightarrow$ \ref{9-5}.
 Suppose that there exists a non-trivial
 $g\in \hat G_0$ such that $\beta_{r,w}(g)=g$ for some
 $w\in\mathtt{Word}(k)$.
Set $w':=\alpha_{f^k}(g)(w)$.
Take a loop $\gamma$ such that $g=[\gamma]$.
We know that the length of $L_{w^n}(\gamma)$ tends to zero as $n\to\infty$,
 and that $[l_{w^n}L_{w^n}(\gamma)l_{{w'}^n}^{-1}]=g$ is a non-trivial loop.
Hence $w'\ne w$, and $(ww\cdots,w'w'\cdots)\in Q$.
\end{proof}

\begin{remark}
  In Lemma \ref{lemma4}, the assumption that $f$ is Cantor is not necessary.
  In fact, we have used only the finiteness of $\hat G_\Omega$.
   Instead, we may use the finiteness of
  $\hat G_{\Omega,L}=\{[\gamma]\in \hat G_\Omega: |\gamma|<L\}$ for $L>0$,
   where $|\gamma|$ is the length  of a path.
   Moreover, for a general hyperbolic rational map $f$,
   we can show $\{[\gamma_q]:q\in Q\}
   =\bigcap_{k>0}\bigcup_{w\in\mathtt{Word}(k)}\beta_{r,w}(\hat G_{\Omega,L})$
   for some $L>0$.

For another version of this lemma, see Section 7 of \cite{Ka06}.
\end{remark}

\subsection{Proof}

Let $$f(z)=\frac{az^4-2az^2+a+\dfrac1{4a}}{z^2-1}=a(z^2-1)+\frac1{4a(z^2-1)}.$$
Then $\mathtt{Cr}=\{0,\infty,\pm\sqrt{1+1/{2a}},\pm\sqrt{1-1/{2a}}\}.$
$\pm\sqrt{1+1/{2a}}\mapsto1\mapsto\infty\mapsto\infty,
\pm\sqrt{1-1/{2a}}\mapsto-1\mapsto\infty$.

\begin{lemma}
 If $|a|$ is large enough, then $f$ is d-Cantor and $f^2$ is s-Cantor.
\end{lemma}

\begin{proof}
 Suppose $|a|>2$.
 Then it is easily seen that $|f(z)|>|a||z|$ provided $|z|>5/3$.
 Hence $\{|z|>5/3\}$ is included in $A_0(\infty)$,
 the immediate basin of attraction of $\infty$.
 Thus $f(0)=-a-\frac1{4a}\in A_0(\infty)$.

Note that $f=g\circ h$ with $g(z)=\frac{z+1/z}2,h(z)=2a(z^2-1)$.
It is easy to see that $g^{-1}([-1,1])=\{|z|=1\}$.
Hence $f^{-1}([-1,1])$ is a union of two simple closed
curves, say $\Gamma_1,\Gamma_2$, which encircle
 $-1$ and $1$ separately.
 Observe $-\sqrt{1\pm1/2a}\in \Gamma_1,\sqrt{1\pm1/2a}\in \Gamma_2$.

Set $D=\{|z|<5/3\}$.
Then $f^{-1}(D)\subset D$ is the union of two annuli, which
encircle two critical values $\pm1$ separately;
$f^{-2}(D)\subset f^{-1}(D)$ is
the union of four annuli,
which encircle four critical points $\pm\sqrt{1\pm1/2a}$ separately.
Two critical values $\pm1$ are included in the unbounded component
of $\mathbb C-f^{-2}(D)$.
Thus there exists a topological disc $D'\subset D$ such that
$\pm1\notin \overline{D'}$ and $f^{-2}(D)\subset D'$, which
satisfies the property of being s-Cantor.
This completes the proof.
\end{proof}

\begin{theorem}\label{th3}
Suppose that $|a|$ is large enough for $f^2$ to be s-Cantor.
Then $f$ is not t-Cantor, namely, for any radial $r$ for $f$,
$\phi_r:\Sigma_4\to J$ is not one-to-one.
\end{theorem}

\begin{proof}
First we show that for some radial $r$ and
a generating set $\{A,B,C_k \ (k=0,1,2,\dots)\}\subset G_0$,
the homomorphism
$f_r^*:G_0\to G_1$ has the following description:
\begin{eqnarray}
A&\mapsto& A\cdot(1,2)+A^{-1}\cdot(2,1)+(3,4)+(4,3)\label{eq9}
\\
B&\mapsto& (1,2)+(2,1)+B\cdot(3,4)+B^{-1}\cdot(4,3)\label{eq10}
\\
C_0&\mapsto& B\cdot(1,4)+(2,2)+(3,3)
+B^{-1}\cdot(4,1)\label{eq11}\\
C_k&\mapsto&
(1,1)+(2,2)+(3,3)+C_{k-1}\cdot(4,4),
\ (k\ge1),\label{eq12}
\end{eqnarray}
where we simply write $(i,j)$ instead of $e\cdot(i,j)$.

To this end, let $a=\sqrt{-1}c$ with $c>0$ large enough.
The choice of the parameter is not important.
We fix the parameter to avoid complicated notation.
 Note that $f^k(0)\in\sqrt{-1}\mathbb R_{\leq0},k=1,2,\dots$ with
 $0>\mathrm{Im}\,f(0)>\mathrm{Im}\,f^2(0)>\cdots$ and
 $$\begin{array}{ccc}
 f^{-1}(\sqrt{-1}\mathbb R)&=&\mathbb R\cup\sqrt{-1}\mathbb R,\\
 f^{-1}(\sqrt{-1}\mathbb R_{\ge0})
 &=&J_1\cup J_2\cup J_3\cup J_4\subset\mathbb R,\end{array}$$ where
 $$\begin{array}{ll}
 J_1=(-\infty,-\sqrt{1+1/2c}],&J_2=(-1,-\sqrt{1-1/2c}],\\
 J_3=[\sqrt{1-1/2c},1),&J_4=[\sqrt{1+1/2c},\infty).
\end{array}$$
 We also set $\bar x=\sqrt{-1}$.
Then we have $f^{-1}(\bar x)=\{x_1,x_2,x_3,x_4\}\subset\mathbb R$
with $x_i\in J_i,i=1,2,3,4$.
Recall that we have two loops $\Gamma_1$ and $\Gamma_2$, which are
the inverse image of $[-1,1]$ encircling $-1$ and $1$ respectively.
Observe that the graph $L:=\bigcup_{i=1}^4 J_i\cup\bigcup_{i=1}^2 \Gamma_i
=f^{-1}(\sqrt{-1}\mathbb R_{\ge0}\cup[-1,1])$ divides $\mathbb C$ into
three connected components; one is unbounded and the other two are
bounded.

Take a radial $r=(l_i)_{i=1}^4$
such that $l_i$'s are arcs joining $\bar x$ and $x_i$
in the upper half-plane.
The group $G_0$ is  generated by $A=[\gamma_A],B=[\gamma_B],
C_k=[\gamma_{C_k}],k=0,1,\dots$,
 where $\gamma_A,\gamma_B,\gamma_{C_k}$
are loops in $\hat{\mathbb C}-P$ counterclockwise encircling
$-1,1,f^{k+1}(0)$ respectively.
Specifically, we define $\gamma_A$ (resp. $\gamma_B$) as
a loop obtained by going around the segment $I_1$ (resp. $I_2$)
joining $\bar x$
and $-1$ (resp. $1$), and $\gamma_{C_0}$ as a loop obtained by
going around a simple arc $I_3$ joining $\bar x$ and $f(0)$
in $\mathbb C-(\{x\le1\}\cup \{\sqrt{-1}y:y\le\mathrm{Im}\,f^2(0)\})$.
Finally we define
$\gamma_{C_{k+1}}$ as a loop  containing no critical point such that
$[\gamma_{C_{k+1}}]=[f(l_4^{-1}\gamma_{C_k}l_4)], k=0,1,\dots$, and (\ref{eq12})
 holds.

 Observe that $f^{-1}(I_1)$ has two connected components
 $I_{11}\subset \{\mathrm{Im}\, z\le0\}$ joining $x_1$ and $x_2$, and
 $I_{12}\subset \{\mathrm{Im}\,z\ge0\}$ joining $x_3$ and $x_4$, and that
 $l_1\cup I_{11}\cup l_2$ is homotopic to $\gamma_{A}$ in $\hat{\mathbb C}-P$,
 and $l_3\cup I_{12}\cup l_4$ is a trivial loop in $\hat{\mathbb C}-P$.
 A similar argument is true for $I_2$.
 Now we obtain (\ref{eq9}) and (\ref{eq10}).
The inverse image $f^{-1}(I_3)$ has three connected components, one of which,
say $I_{31}$, joins $x_1$ and $x_4$ included in the unbounded connected component
of $\mathbb C-L$.
It is easily seen that $l_1\cup I_{31}\cup l_4$ is homotopic to $\gamma_B$.
Thus we have (\ref{eq11}).

Let $T$ be the image of $\alpha:\pi_1(\hat{\mathbb C}-P,\bar x)\to
\mathfrak{S}(\mathtt{Word}(1))$ defined in Definition \ref{defi8}.
Recall $\mathtt{Word}(1)=\{1,2,3,4\}$.
Then
$$\begin{array}{ccl}
T&=&\langle(1\,4),(1\,2)(3\,4)\rangle\\
&=&\{\mathrm{id},(1\,4),(2\,3),(1\,2)(3\,4), (1\,3)(2\,4),
(1\,4)(2\,3),(1\,2\,4\,3),(1\,3\,4\,2)\}
\end{array}$$
 (we write a permutation
in the form of a product of cycles).
Notice that the product of $\tau,\tau'\in T$ is $\tau\tau'=\tau'\circ\tau$.
The group $T$ acts on $\mathtt{Word}(1)$ from the right by
$i\cdot\tau=\tau(i)$ and on
${\mathbb Z_2}^4=\mathrm{Map}(\mathtt{Word}(1),{\mathbb Z_2})$ from the left
 by $\tau\cdot q=q\circ \tau$, where
$q\in{\mathbb Z_2}^4,\tau\in T,i\in \mathtt{Word}(1)$.

The permutation group $T$ is characterized as the subgroup
of $\mathfrak{S}(\mathtt{Word}(1))$ which preserves
the partition $\mathcal{P}=\{\{1,4\},\{2,3\}\}$, namely,
$$T=\{\tau\in \mathfrak{S}(\mathtt{Word}(1)):P\cdot\tau\in\mathcal{P}\text{ for }P\in\mathcal{P}\}.$$
Let $L$ be the subgroup fixing each block of $\mathcal P$, namely,
$$\begin{array}{ccl}
L&=&\{\tau\in \mathfrak{S}(\mathtt{Word}(1)):P\cdot\tau=P\text{ for }P\in\mathcal{P}\}\\
&=&\{\mathrm{id},(1\,4),(2\,3),(1\,4)(2\,3)\}.
\end{array}$$

The action of $T$ keeps
the following two subgroups unchanged:
 $$
\begin{array}{ccl}
Q&:=&\ker((q_1,q_2,q_3,q_4)\mapsto \sum_{i=1}^4q_i)\\
&=&\{(q_1,q_2,q_3,q_4)\in{\mathbb Z_2}^4\,|\,
\sum_{i=1}^4q_i=0\},\\
S&:=&\ker((q_1,q_2,q_3,q_4)\mapsto (q_1+q_4,q_2+q_3))\\
&=&\{(q_1,q_2,q_3,q_4)\in{\mathbb Z_2}^4\,|\,
q_1=q_4,q_2=q_3\}\\
&=&\{(0,0,0,0),(1,0,0,1),(0,1,1,0),(1,1,1,1)\}\\
&=&\{q\in{\mathbb Z_2}^4\,|\,\tau\cdot q=q \text{ for }\tau\in L\}.
\end{array}$$
The group $T$ acts on the quotient groups $\mathbb Z_2^4/Q$ and
$Q/S$ trivially (not on $\mathbb Z_2^4/S$), and the group
$M({{\mathbb Z}_2},\mathtt{Word}(1))$ has subgroups
${\mathbb{Z}_2}^4\rtimes T, Q\rtimes T,
S\rtimes T$.

Note that the product of $(q,\tau),(q',\tau')\in
M({{\mathbb Z}_2},\mathtt{Word}(1))$ is
$(q,\tau),(q',\tau')=(q+\tau\cdot q',\tau\tau')$, and
the inverse of $(q,\tau)$ is $(q,\tau)^{-1}=(\tau^{-1}\cdot q,\tau^{-1})$.

\medskip

Claim 1.
The following hold.
\begin{enumerate}
  \item For $q\in {\mathbb{Z}_2}^4,\tau\in T$,
  $$q+\tau\cdot q\in S\iff \tau\in L\text{ or }q\in Q.$$
  \item ${\mathbb{Z}_2}^4\rtimes T$ has normal subgroups
$Q\rtimes T$, $Q\rtimes L$, and $S\rtimes L$.
\item The quotient group $R:=(Q\rtimes T)/(S\rtimes L)$ is
isomorphic to the Klein four-group $D_2\cong \mathbb Z_2\times\mathbb Z_2$.
\item Let $R=\{e,\zeta_1,\zeta_2,\zeta_3\}$.
Here $e=S\rtimes L$ is the identity element,
$\zeta_1=(S\rtimes T)-(S\rtimes L), \zeta_2=(Q\rtimes L)-(S\rtimes L)$,
and $\zeta_3=\zeta_1\zeta_2$.
Then $\zeta_1$ and $\zeta_3$ are conjugate in
$({\mathbb{Z}_2}^4\rtimes T)/(S\rtimes L)$.
\end{enumerate}

{\em Proof of Claim 1.}

(1)
Let $q=(q_1,q_2,q_3,q_4)$.
Then
$q+\tau\cdot q
=(q_1+q_{\tau(1)},q_2+q_{\tau(2)},q_3+q_{\tau(3)},q_4+q_{\tau(4)})$.
We have $q+\tau\cdot q\in S$ if and only if
$q_1+q_{\tau(1)}=q_4+q_{\tau(4)},q_2+q_{\tau(2)}=q_3+q_{\tau(3)}$.
This is equivalent to
\begin{equation}\label{eq:8-1}
    q_1+q_4=q_{\tau(1)}+q_{\tau(4)},q_2+q_3=q_{\tau(2)}+q_{\tau(3)},
\end{equation}
 which holds if $\tau\in L$.
For $\tau\in T-L$, we know that
 $\{1,4\}\cdot\tau=\{2,3\},\{1,4\}\cdot\tau=\{2,3\}$.
Thus if $\tau\in T-L$, (\ref{eq:8-1}) is equivalent to
$q_1+q_4=q_2+q_3$, the equation which the elements of $Q$ satisfy.

(2)
If $(q,\tau)\in {\mathbb{Z}_2}^4\rtimes T,(s,\tau')\in S\rtimes L$, then
we have $(q,\tau)(s,\tau')(q,\tau)^{-1}=
(q+\tau\cdot s+\tau\tau'\tau^{-1}\cdot q,\tau\tau'\tau^{-1})
\in S\rtimes L$ from $\tau\cdot s\in S$, $\tau\tau'\tau^{-1}\in L$, and
the fact $q+\sigma\cdot q\in S$ for $\sigma\in L$, which we have proved above.
Thus ${\mathbb{Z}_2}^4\rtimes T\vartriangleright S\rtimes L$.
A similar argument implies ${\mathbb{Z}_2}^4\rtimes T\vartriangleright Q\rtimes L$.

From the index $|{\mathbb{Z}_2}^4\rtimes T:Q\rtimes T|=2$,
it follows that ${\mathbb{Z}_2}^4\rtimes T\vartriangleright Q\rtimes T$.

(3)
The order $|R|=4$.
There exist two subgroups $S\rtimes T,Q\rtimes L$ of $Q\rtimes T$
greater than $S\rtimes L$.
Thus $R$ is not the cyclic group of order 4, and so $R\cong D_2$.

(4)
The element $\zeta_2$ is conjugate to only itself, since
${\mathbb{Z}_2}^4\rtimes T\vartriangleright Q\rtimes L$.
To show that $\zeta_1$ is conjugate to $\zeta_3$,
it is sufficient to show that $S\rtimes T$ is not a
normal subgroup of ${\mathbb{Z}_2}^4\rtimes T$.
Take $(p,\sigma)\in {\mathbb{Z}_2}^4\rtimes T$ with
$p\in {\mathbb{Z}_2}^4-Q$, and $(s,\tau)\in S\rtimes T$
with $\tau\in S-L$.
Then $(p,\sigma)(s,\tau)(p,\sigma)^{-1}=
(p+\sigma\cdot s+\sigma\tau\sigma^{-1}\cdot p,\sigma\tau\sigma^{-1})
\notin S\rtimes T$, since $p+\sigma\tau\sigma^{-1}\cdot p\notin S$ by (1).
This completes the proof of Claim 1.


\begin{remark}
$({\mathbb{Z}_2}^4\rtimes T)/(S\rtimes L)$ is isomorphic to $D_4$ the dihedral
  group of order 8.
\end{remark}

By the definition of $R$,
\begin{equation}\label{eq:9-1}
  \begin{array}{ccc}
(q,\tau)\in \zeta_1\text{ or } \in \zeta_3
&\Rightarrow&\tau\ne\mathrm{id}\\
(q,\tau)\in \zeta_2\text{ or } \in \zeta_3
&\Rightarrow& q\ne0
\end{array}
\end{equation}
This property is
the key of the proof of Theorem \ref{th3}.

An element $g\in G_0$ is expressed as
$$g=t_1^{n_1}t_2^{n_2}\cdots t_m^{n_m},
t_i\in \{A,B,C_k\,(k=0,1,2\dots)\},n_i\in\mathbb Z.$$
We write $|g|_A=\sum_{t_i=A}n_i$ and
$|g|_B=\sum_{t_i=B}n_i$.
Consider homomorphisms $\mu_A:G_0
\to \mathbb Z_2$ and $\mu_B:G_0
\to \mathbb Z_2$ by $\mu_A(g)=|g|_A\mod 2$ and
$\mu_B(g)=|g|_B\mod 2$.

The homomorphisms $\mu_A$ and $\mu_B$
 are naturally extended to $\mu_A,\mu_B:G_k
\to M(\mathbb Z_2,\mathtt{Word}(k))$.
It is easily seen that the image of $f_r^*:G_0\to G_1$ is included in
${G_0}^4\rtimes T$, and that the images of $\mu_A\circ f_r^*,\mu_B\circ f_r^*:
G_0\to {\mathbb Z_2}^4\rtimes T$ are included in $Q\rtimes T$.

Consider $\mu_A\times\mu_B:f_r^*(G_0)\to
({\mathbb Z_2}^4\rtimes T)\times({\mathbb Z_2}^4\rtimes T)$,
and let $\pi:({\mathbb Z_2}^4\rtimes T)\times ({\mathbb Z_2}^4\rtimes T)\to
D_4\times D_4$ be the projection.
Then $\pi\circ (\mu_A\times \mu_B)\circ f_r^*$ sends
$$
\begin{array}{ll}
  A&\mapsto ((1,1,0,0),(12)(34))\times((0,0,0,0),(12)(34))\mapsto (\zeta_3,\zeta_1)\\
  B&\mapsto ((0,0,0,0),(12)(34))\times((0,0,1,1),(12)(34))\mapsto (\zeta_1,\zeta_3)\\
C_0&\mapsto ((0,0,0,0),(14))\times((1,0,0,1),(14))\mapsto (e,e)\\
C_k&\mapsto ((0,0,0,0),\mathrm{id})\times((0,0,0,0),\mathrm{id})\mapsto (e,e), k=1,2,\dots
\end{array}
$$
Since the  the kernel of $\mu_A\oplus\mu_B:G_0\to \mathbb Z_2\oplus
\mathbb Z_2$ is the normal subgroup of $G_0$ generated by
$$\{A^2,B^2,(AB)^2,C_k|k=0,1,\dots\},$$
we obtain the homomorphism
$\tilde f_r:\mathbb Z_2\oplus\mathbb Z_2\to (\mathbb Z_2^4\rtimes T/S\rtimes L)
\times (\mathbb Z_2^4\rtimes T/S\rtimes L)\cong D_4\times D_4$
reduced from $f_r^*:{G_0}\to {G_0}^4\rtimes T\subset G_1$:
$$
  \label{eq:commutative-diagram}
  \raisebox{-0.5\height}{\includegraphics[trim=0 710 300 20,clip]{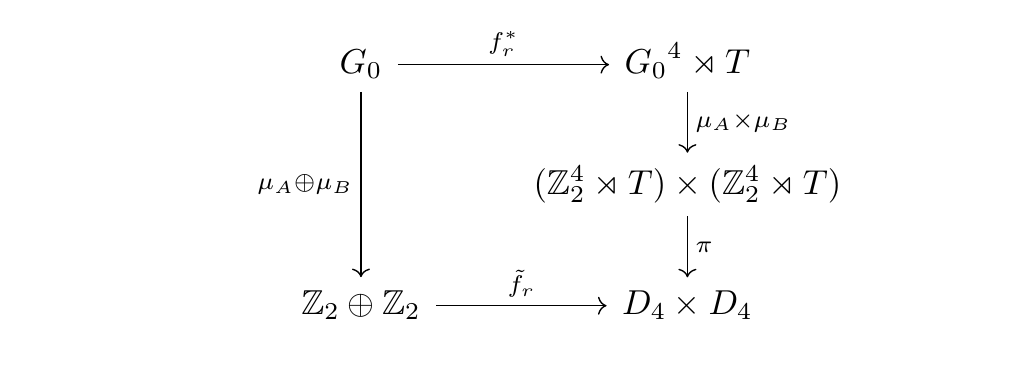}}
$$
The image of $\tilde f_r$ is included in $(Q\rtimes T/S\rtimes L)\times (Q\rtimes T/S\rtimes L)= R\times R$, and
$\tilde f_r$ sends $(1,0)$ to $(\zeta_3,\zeta_1)$, $(0,1)$ to $(\zeta_1,\zeta_3)$,
(and so $(1,1)$ to $(\zeta_2,\zeta_2)$).

Let $r'=(l_i')$ be a radial for $f$ other than $r$.
Then $f_{r'}^*(g)=hf_{r}^*(g)h^{-1}$ for some
$h\in {G_0}^4\rtimes\{\mathrm{id}\}$ as seen in
Section \ref{sec6-1}.
From $\mathbb Z_2^4\rtimes T\vartriangleright Q\rtimes T$,
we also have a similar commutative diagram as above for $f_{r'}^*$, namely,
we reduce $f_{r'}^*$ to
$\tilde f_{r'}:\mathbb Z_2\oplus\mathbb Z_2\to D_4\times D_4$, which is
obtained by
$$\tilde f_{r'}(x)
=(\mu_A(h)k_1\mu_A(h^{-1}),\mu_B(h)k_2\mu_B(h^{-1})),$$ where
$\tilde f_r(x)=(k_1,k_2)$.
By 4 of Claim 1, we have only four possibilities of the behavior of $\tilde f_{r'}$:
\begin{itemize}
  \item[Case 1.] $(1,0)\mapsto (\zeta_1,\zeta_3),\ (0,1)\mapsto (\zeta_3,\zeta_1)$
  \item[Case 2.] $(1,0)\mapsto (\zeta_3,\zeta_3),\ (0,1)\mapsto (\zeta_1,\zeta_1)$
  \item[Case 3.] $(1,0)\mapsto (\zeta_1,\zeta_1),\ (0,1)\mapsto (\zeta_3,\zeta_3)$
  \item[Case 4.] $(1,0)\mapsto (\zeta_3,\zeta_1),\ (0,1)\mapsto (\zeta_1,\zeta_3)$
\end{itemize}
Now (\ref{eq:9-1}) means:
\begin{equation}\label{eq:10}
  \begin{array}{ccc}
\mu_X\circ f_{r'}^*(g)\in \zeta_1\text{ or } \in \zeta_3
&\Rightarrow&\alpha_f(g)\ne\mathrm{id}\\
\mu_X\circ f_{r'}^*(g)\in \zeta_2\text{ or } \in \zeta_3
&\Rightarrow& \exists i\in\mathtt{Word}(1),\mu_X\circ\beta_{r',i}(g)=1,
\end{array}
\end{equation}
where $X\in\{A,B\}$.

\medskip

Claim 2.
In each case, we can find $g\in G_0$ such that
for any $k>0$ there exists $w\in\mathtt{Word}(k)$  satisfying
$\beta_{r',w}(g)\notin\ker(\mu_A\oplus\mu_B)$.

{\em Proof of Claim 2.}\

Case 1 and 4.
By (\ref{eq:10}),
$$\begin{array}{cl}
(\mu_A\oplus\mu_B)(g)=(1,0) \Rightarrow&
\exists i\in\mathtt{Word}(1),\ (\mu_A\oplus\mu_B)(\beta_{r',i}(g))\ne(0,0)\\
(\mu_A\oplus\mu_B)(g)=(0,1) \Rightarrow&
\exists i\in\mathtt{Word}(1),\ (\mu_A\oplus\mu_B)(\beta_{r',i}(g))\ne(0,0)\\
(\mu_A\oplus\mu_B)(g)=(1,1) \Rightarrow&
\exists i\in\mathtt{Word}(1),\ (\mu_A\oplus\mu_B)(\beta_{r',i}(g))\ne(0,0)
\end{array}$$

Hence any $g\notin\ker(\mu_A\oplus\mu_B)$ satisfies the required property.

Case 2.
By (\ref{eq:10}),
$$\begin{array}{cl}
(\mu_A\oplus\mu_B)(g)=(1,0) &\Rightarrow
\exists i\in\mathtt{Word}(1),\ \beta_{r',i}(g)\in \{(1,0),(1,1)\}\\
(\mu_A\oplus\mu_B)(g)=(1,1) &\Rightarrow
\exists i\in\mathtt{Word}(1),\ \beta_{r',i}(g)\in\{(1,0),(1,1)\}.
\end{array}$$
Hence any $g\in(\mu_A\oplus\mu_B)^{-1}(\{(1,0),(1,1)\})$ satisfies the required property.

The remaining case is similar to Case 2.
This completes the proof of Claim 2.

\medskip

Claim 3.
The homomorphism $\mu_A\oplus\mu_B:G_0\to\mathbb Z_2\oplus\mathbb Z_2$
breaks up
 into $G_0\overset{\alpha_\infty}{\rightarrow} \hat G_0\rightarrow\mathbb Z_2\oplus\mathbb Z_2$.

{\em Proof of Claim 3.}
We return to the original radial $r$, whose behavior is Case 4.

First note that if $g\in \ker \alpha_f$, then
$\pi\circ(\mu_A\times\mu_B)\circ f_r^*(g)\ne
(\zeta_1,\zeta_3),(\zeta_3,\zeta_1)$ by (\ref{eq:10}), and so $(\mu_A\oplus\mu_B)(g)=(0,0)$ or $(1,1)$.

We show $\ker\alpha_\infty\subset\ker(\mu_A\oplus\mu_B)$.
Take $g\in\ker\alpha_\infty$, and assume $(\mu_A\oplus\mu_B)(g)\ne(0,0)$.
Then $g_w:=\beta_{r,w}(g)\in\ker\alpha_f$ and
\begin{equation}\label{eq:11}
  (\mu_A\oplus\mu_B)(g_w)=(0,0)\text{ or }(1,1)
  \end{equation}  for any $w$;
by Claim 2, for any $k\ge0$ there exists $w(k)\in\mathtt{Word}(k)$
such that
\begin{equation}\label{eq:13}
  (\mu_A\oplus\mu_B)(g_{w(k)})=(1,1).
  \end{equation}
By (\ref{eq9}) to (\ref{eq12}), there exists $k_0\in\mathbb N$ such that
if $k\geq k_0$ and  $w\in \mathtt{Word}(k)$, then
$g_w\in\mathcal H:=\langle A,B\rangle$.
Fix $k\ge k_0$.
From (\ref{eq:11})
$$g_w\in \mathcal H_0:=\langle A^2,B^2,AB\rangle=\mathcal H\cap
(\mu_A\oplus\mu_B)^{-1}(\{(0,0),(1,1)\}).$$
We can deduce a contradiction that
$g_{iw(k)}\notin\mathcal H_0$ for $i\in\mathtt{Word}(1)$.
Indeed, let $\mathcal H_1$ be the normal subgroup of $\mathcal H$ generated by $\{A^2,B^2\}$.
Then $\mu_A(\mathcal H_1)=\mu_B(\mathcal H_1)=0$, and $f_r^*(\mathcal H_1)=\{e\}$.
The quotient group $\mathcal H_0/\mathcal H_1$ is the cyclic group generated by $\{AB\mathcal H_1\}$.
By (\ref{eq:13}),  $g_{w(k)}\in (AB)^m\mathcal H_1$ with $m$ odd.
Observe $f_r^*(g_{w(k)})=f_r^*((AB)^m)=A^m(1,1)+A^{-m}(2,2)+B^m(3,3)+B^{-m}(4,4)$.
Thus the proof of Claim 3 is completed.

We have the commutative diagram:
$$
\raisebox{-0.5\height}{\includegraphics[trim=0 710 300 10,clip]{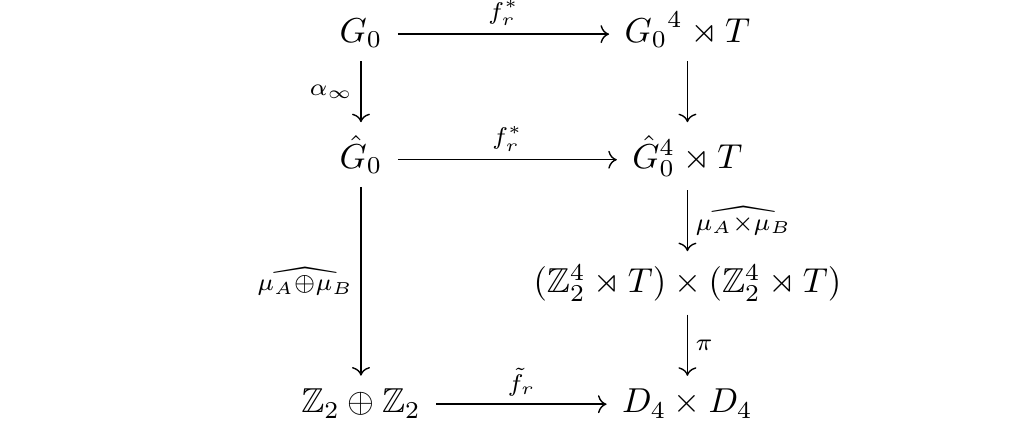}}
$$

Now for any radial $r'$, there exists $\hat g\in\hat G_0$ such that
for any $k>0$ there exists $w\in\mathtt{Word}(k)$ satisfying
$\beta_{r',w}(\hat g)\ne e$, where $e$ is the identity element of $\hat G_0$.
Indeed, let $g$ be as in Claim 2, and $\hat g=\alpha_\infty(g)$.
From Lemma \ref{lemma4}, $\phi_{r'}$ is not one-to-one.

\end{proof}

\begin{theorem}\label{th10}
 There exists $\bar f\in S_d-T_d$ for $d\ge 4$.
\end{theorem}

We use the Thurston type theorem for a characterization of subhyperbolic
rational maps.

\begin{defi}
  Let $g:\hat{\mathbb C}\to\hat{\mathbb C}$ be a topological branched covering.
  Let $P_g$ be the set of the points in the critical orbits.
  We say $g$ is  a {\em subhyperbolic semi-rational} map
  if $\overline P_g-P_g$ is a finite set, $g$ is holomorphic in a
  neighborhood of $\overline P_g-P_g$, and every periodic points
  in $\overline P_g-P_g$ is (super)attracting.

  We say two subhyperbolic semi-rational maps $g_1,g_2$ are {\em c-equivalent} to
  each other if there exist two homeomorphisms
  $\phi,\psi:\hat{\mathbb C}\to\hat{\mathbb C}$ satisfying:
  (1) $\phi\circ g_1=g_2\circ \psi$, (2) $\phi=\psi$ and is holomorphic
  on a neighborhood $U$ of  $\overline{P_{g_1}}-P_{g_1}$, (3) $\phi=\psi$
  on $P_{g_1}$, and (4) $\phi$ and $\psi$ are isotopic to each other
  rel $P_{g_1}\cup U$.
\end{defi}

\begin{theorem}\label{theoremtct}(Zhang-Jiang \cite{ZhangJiang09},Cui-Tan \cite{CuiTan11})
Let $g$ be a subhyperbolic semi-rational map with $\#P_{g}=\infty$.
The map $g$ is c-equivalent to a rational map
 if and only if $g$ has no Thurston obstruction.
\end{theorem}

\begin{cor}\label{cor1}
  If $g$ is a subhyperbolic semi-rational map without
  homotopically invariant curve, then $g$ is c-equivalent to
  a hyperbolic rational map whose Julia set is Cantor.
\end{cor}

\begin{proof}[Proof of Theorem \ref{th10}]
 We construct $\bar f$ by modifying $f$ above.
 Let $a=ic$ with $c>0$ large again.
 Then $[\sqrt{1+1/2y},\infty)$ is included in $A_0(\infty)$ the attracting basin
 of $\infty$.
 Take two points $z_1,z_2\in \mathbb R$ with $\sqrt{1+1/2y}<z_1<z_2$
 and $d-3$ smooth functions $h_i:[z_1,z_2]\to\mathbb R,i=1,2,\dots,d-3$
 such that $h_i(z_1)=h_i(z_2)=0$ ($i=1,2,\dots,d-3$),
 $h_i(x)<h_{i+1}(x)$ ($z_1<x<z_2, i=1,2,\dots,d-4$),
 and $\{x+iy:z_1\le x\le z_2,h_1(x)\le y\le h_{d-3}(x)\}\subset A_0(\infty)$.
 Let $E_i=\{x+iy:z_1\le x\le z_2,h_i(x)\le y\le h_{i+1}(x)\}$
 ($i=1,2,\dots,d-4$).
Define a branched covering $g$ of degree $d$ with critical set
$\mathrm{Cr}_g=\mathrm{Cr}_f\cup \{z_1,z_2\}$ satisfying
 $g(\partial E_i)=f([z_1,z_2])$,
 $g|_{\mathrm{int} E_i}:\mathrm{int} E_i\to \hat{\mathbb C}-f([z_1,z_2])$ is homeomorphic,
 and $g|_{\hat{\mathbb C}-\bigcup_i E_i}=f\circ \beta$,
 where $\beta:\hat{\mathbb C}-\bigcup_i E_i\to\hat{\mathbb C}-[z_1,z_2]$ is a
 homeomorphism which is the identity outside a small
 neighborhood of $\bigcup_i E_i$.

By Corollary \ref{cor1}, there exists a rational map $\bar f$ c-equivalent
to $g$ with $J_{\bar f}$ Cantor.
Let  $r=(l_i)_{i=1,2,\dots,d}$ be a radial for $\bar f$.
If the coding map $\phi_r:\Sigma_d\to J_{\bar f}$ is one-to-one, then
we have a radial $r'$ for $f$ with $\phi_{r'}:\Sigma_4\to J_{f}$ is one-to-one.
Indeed, we can assume that $x_i=l_i(1)\in \mathbb C-\bigcup_i E_i$
for $i=1,2,3,4$ and  $x_i\in E_{i-4}$ for $i=5,6,\dots,d$.
Moreover we can assume that $l_i\subset \mathbb C-\bigcup_i E_i$
for $i=1,2,3,4$, since $\bigcup_i E_i$ does not intersect $P_{\bar f}$.
Define a radial $r'=(l_i')_{i=1,2,3,4}$ for $f$
which is derived from $(l_i)_{i=1,2,3,4}$.
It is easy to see that if $\phi_r$ is one-to-one, so is $\phi_{r'}$.
\end{proof}

\section*{Appendix}
In Appendix, we describe a generalization of Section 3.
We state a version of Theorem \ref{th1} without the assumption
$J\cap\mathtt{Cr}=\emptyset$.
We omit the proof, most of which are parallel to Theorem 1.
The detail is left to the reader.

\begin{defi}
 Let $X$ be a topological space.
  We say that a continuous mapping $a:X\to\hat{\mathbb C}$ is
 {\em tiny}
 if there exist $p\in J$
 and  a homotopy
 $H:X\times[0,1]\to\hat{\mathbb C}- (P-\{p\})$ between $a$ and a constant
 mapping $p$.
 In particular, a homotopically trivial mapping is tiny.
 We say that $p$ is the {\em center} of $a$.
If $a$ is tiny and homotopically nontrivial, then
the choice of $p$ is unique and  $p\in P$.

 A subset $X\subset\hat{\mathbb C}$ is said to be
 tiny if the inclusion map is tiny.
\end{defi}

\begin{defi}
Set $P'=\overline P- \bigcup_{k=0}^\infty f^{-k}(\mathtt{Pa})$.
 Let $\Omega$ be an expanding domain.

A closed curve $\gamma:S^1\to \Omega-P'$
is {\em homotopically invariant} if $\gamma$ is
homotopically nontrivial, and
  there exists a lifted loop $\gamma':S^1\to \Omega-P'$ of $\gamma$
  by $f^k$  for some $k>0$ and there
  exists a homotopy $H:S^1\times[0,1]\to \Omega-P'$
  between $\gamma$ and $\gamma'$ such that for $s\in S^1$,
 if $H(s,t_0)\in P$ for some $t_0\in[0,1)$, then $H(s,t)=H(s,t_0)$
 for any $t\in[0,1]$.

\end{defi}

\begin{defi}
For a subset $U\subset \hat{\mathbb C}$ and $p\in P$, we say that
 $p$ is {\em weakly contained} in $U$ if either $p\in U$ or $p$ is
 the center of  some tiny topological disc $D$
 satisfying $p\in D$ and $\partial D\subset U$.
\end{defi}

\begin{theorem}\label{th0.0}
 Let $f$ be a geometrically finite rational map of degree $d$.
  Take  an expanding domain   $\Omega$.
The following are equivalent.
\begin{enumerate}
\item\label{00-1} $f$ is Cantor.
\item\label{00-2}
     Every connected component of $J$ is tiny.
 \item\label{00-2.5} $F\ne\emptyset$, and every homotopically invariant  curve
       is tiny.
 \item\label{00-3}
      There exists  $n>0$ such that for any  closed curve
      $\gamma\subset\Omega-P$,  any lifted loop of $\gamma$ by $f^{-n}$
       is tiny.
\item\label{00-4}
There exists  $n>0$ such that  any connected component of
$f^{-n}(\Omega)$ is tiny.

\item\label{00-6}
     There exists a solitary fixed point $p$ such that every postcritical
     value is  weakly
       contained in $A_0(p)$.

\item\label{00-7}
     There exists a solitary fixed point $p$, and
     every critical value is weakly contained in
     the basin of attraction $A(p)$, and
      at least $2d-4$ critical values counted with
     multiplicity are weakly contained in the immediate basin of attraction
     $A_0(p)$.
      \end{enumerate}
\end{theorem}

\begin{theorem}
 Let $f$ be a geometrically finite rational map with
 $\mathtt{Pa}\cap P=\emptyset$.
  Take  an expanding domain   $\Omega$.
The following are equivalent.
\begin{enumerate}
\item\label{c0-1} $f$ is Cantor.
 \item\label{c0-5.0}
      The quotient group
      $\pi_1(\Omega-P,\bar x)/\ker \alpha_{\infty}$
      is a finite group.
 \item \label{c0-3}
       There exists a neighborhood $K$ of $J$ such that
       $K-P$ is arcwise-connected and
       $\pi_1(K-P,\bar x)/\ker \alpha_{\infty}$
      is a finite group.
\end{enumerate}
\end{theorem}

\begin{remark}
 Let $f$ be a geometrically finite rational map.
 If $J$ is Cantor, then $F$ contains at least two critical values.
 Indeed, for a solitary fixed point $p$, take
 $U_k$ as in Definition \ref{defidomain}.
 If every $U_k$ has at most one critical value, then
 every $U_k$ is simply connected, and so $J$ is not Cantor.

 Conversely, for $d\ge 2$, there exists a rational map of degree $d$ such that
 $J$ is Cantor and $F$ contains exactly two critical values of degree two.
\end{remark}

 \begin{remark}
   In \ref{00-2.5} of Theorem \ref{th0.0}, the condition
   $F\ne\emptyset$ is necessary.
  We give an example of a subhyperbolic rational map with
  $F=\emptyset$ such that every  homotopically
  invariant curve $\gamma$ in $\hat{\mathbb C}-P$ is tiny.

  Let $\alpha=1+2i$ and define $g:\mathbb C\to\mathbb C$ by $g(z)=\alpha z$.
  Consider the quotient space $S=\mathbb C/\Gamma$, where
  $\Gamma=\langle z\mapsto -z,z\mapsto z+1,z\mapsto z+i\rangle$.
  Then  $g$ descends to
  $\bar f:S\to S$ by $\bar f([z])=[g(z)]$.
  We have a bijection $h:S\to \hat{\mathbb C}$ with singularity
  at $[0],[1/2],[i/2],[(1+i)/2]$ such that $f=h\circ \bar f
  \circ h^{-1}$ is
  a rational map.
  The map $h$ is induced from the Weierstrass $\wp$ function for
  the lattice $\langle 1,i\rangle$.
  For example, we have
  $f(z)=z(z^2-\alpha)^2/(\alpha z^2-1)^2  $.
  We can see that $P=\{0,\pm1,\infty\}$ consists of fixed points with
  multiplier $\alpha^2$, and $J=\hat{\mathbb C}$.
 Then every  homotopically
  invariant curve $\gamma$ in $\hat{\mathbb C}-P$ is tiny.
\end{remark}

\bibliographystyle{plain}

\end{document}